%% file: main.tex
\documentclass[12pt, oneside, reqno]{amsart}
\input{packages.tex}
\usepackage{stmaryrd}
\definecolor{B}{rgb}{0,0,1}

\definecolor{G}{rgb}{0,1,0}

\title[HDG Methods for the two-dimensional Vector Laplacian]{HDG Methods for the two-dimensional Vector Laplacian}
\author[B. Cockburn]{Bernardo Cockburn}
\address[1]{School of Mathematics, University of Minnesota, Minneapolis, MN 55455, USA}
\email{bcockbur@umn.edu}
\author[C. N\'unez]{Cristhian  N\'uñez}
\address[2]{School of Mathematics, University of Minnesota, Minneapolis, MN 55455, USA}
\email{cnunezra.umn.edu}
\author[M. A. S\'anchez]{Manuel A. S\'anchez}
\address[3]{Instituto de Ingenier\'ia Matem\'atica y Computacional, Facultad de Matem\'aticas y Escuela de Ingenier\'ia, Pontificia Universidad Cat\'olica de Chile, Santiago, Chile.}
\email{manuel.sanchez@uc.cl}
\thanks{C. Nuñez was supported by Doctorado Nacional Scholarship 2021, ANID Chile.}
\thanks{M.A. S\'anchez was partially supported by FONDECYT Regular grant N. 1221189 and by Centro Nacional de Inteligencia Artificial CENIA, FB210017, Basal ANID Chile.}
\date{\today}                        
\begin{document}
\begin{abstract}
We introduce new hybridizable discontinuous Galerkin (HDG) methods for solving the two-dimensional vector Laplacian equation under three types of boundary conditions: electric, magnetic, and Dirichlet. The method is formulated on a first-order system form of the equations, in which the rotational and divergence of the electric field are introduced as auxiliary variables. We study the well-posedness of the method and prove that, when using piecewise polynomial approximations of degree $k \geq 0$, the error in the $L^2$ norm of the electric field converges at the optimal rate of $k+1$.
Additionally, we prove that the $L^2$-errors of the auxiliary variables, the rotational and divergence, converge with order $k + 1/2$. We also show that the methods can be implemented in three different forms, corresponding to three distinct hybridizations based on the choice of the globally coupled unknowns among the numerical traces defined on the mesh skeleton. Finally, we provide numerical tests that not only validate the theoretical convergence rates but also consistently showcase the optimal convergence across all variables.
\end{abstract}
\maketitle
\section{Introduction}\label{section:Introduction} 
We introduce new HDG methods for the two-dimensional equation for the vector-Laplacian,
\begin{alignat}{4}
\label{eq:vector-laplacian}
\curl\,\rot \,\velocity-\grad \divergence \velocity&=\fb, \quad & \text{ in } \Omega.
\end{alignat}
where $\Omega \subset \Real^2$ is a bounded, polygonal, and simply connected domain with a Lipschitz boundary $\Gamma$. 
The simple connectedness of the domain guarantees the absence of harmonic functions solutions, which ensures that both the rotational and divergence operators have closed ranges and trivial kernels, and implies the existence and uniqueness of the continuous solution, see~\cite[Section~5.1]{11_DArnold}.

The two-dimensional vector Laplacian equation is a fundamental model in various areas of Physics and Engineering. For instance, in fluid dynamics, it models the viscous diffusion of the velocity field within the Stokes and Navier–Stokes equations \cite{23_Arnold}. In electromagnetism, its $\curl \, \rot$ operator is crucial to modeling electromagnetic wave propagation in a medium, as seen in the time-harmonic Maxwell equations \cite{1P_Monk}. Furthermore, in elasticity, it describes the deformation of thin elastic plates under external loads. 

The above problem can be deceptively simple, as it leads to the equation, 
\[
-\Delta \velocity = \fb.
\]
However, these are {\em not} decoupled scalar Laplacian equations, as the boundary conditions dictate the coupling of the vector components.

\begin{table}[h]
\centering
\small
\caption{Boundary conditions on $\Gamma$ for the primal formulation of \eqref{eq:vector-laplacian}.}
\label{tab:boundarycondition}
\renewcommand{\arraystretch}{1.2}
\begin{tabular}{@{}p{3cm}p{2.8cm}p{1.6cm}@{}}
\toprule
\hspace{0.2cm}Electric & Magnetic & Dirichlet \\
\midrule
\hspace{0.2cm}$\divergence \velocity=0$           & $\rot \, \velocity=0$         & \multirow{2}{1.8cm}{$\velocity=0$} \\[0.7ex]
\hspace{0.1cm} $\velocity \cdot \normal^\perp=0$  & $\velocity \cdot \normal=0$   & \\
\bottomrule
\end{tabular}
\end{table}

In Table \ref{tab:boundarycondition} we present the three types of boundary conditions considered. Together with the assumptions on the domain, these conditions guarantee the existence and uniqueness of the solution to problem \eqref{eq:vector-laplacian}, see \textit{e.g.} see~\cite[Theorem~4.8]{11_DArnold}. Electric boundary conditions, see the first column of Table \ref{tab:boundarycondition}, arise in electrostatic models involving dielectric or conducting interfaces \cite{electric}. These conditions require the tangential component of the vector field to vanish along the boundary and the field to remain divergence-free. In this case, the vector field $\velocity$  lies in the space $ \mathring{H}(\mathrm{rot}) \cap H(\mathrm{div}) $. Magnetic boundary conditions, see the middle column in Table \ref{tab:boundarycondition}, are relevant in magneto-mechanical applications \cite{magnetic}. Here, the normal component of the vector field must vanish on the boundary, and $\rot \, \velocity$ must be zero. The solution $ \velocity$ belongs to  $H(\mathrm{rot}) \cap \mathring{H}(\mathrm{div})$.  Note that solving the problem~\eqref{eq:vector-laplacian} with magnetic boundary conditions for $\velocity$ and external force $\fb$ is equivalent to solving the same problem with electric boundary conditions for $\velocity^\perp$ with the right-hand side $-\fb^\perp$. Finally, Dirichlet boundary conditions, third column in Table \ref{tab:boundarycondition}, are common in fluid dynamics to model no-slip boundaries \cite{no-slip}. In this case, the velocity field vanishes entirely on the boundary, and the corresponding solution space is \( \mathring{H}(\mathrm{rot}) \cap \mathring{H}(\mathrm{div}) \).

\subsection*{Historical overview} To better describe the relevance of the material presented here, let us place it as part of the historical development of numerical methods for this equation.

\subsubsection*{Exterior Calculus and mixed methods} Mixed finite element methods for the vector Laplacian \eqref{eq:vector-laplacian} have been proposed and analyzed under the Finite Element Exterior Calculus (FEEC) framework applied to the Hodge Laplacian, specifically, acting on 1-forms;  see, \textit{e.g.}, \cite{42_Arnold,11_DArnold,12_DArnold,10_DArnold}. FEEC provides a systematic approach to constructing stable finite element methods that preserve essential differential structures in the De Rham complex. In~\cite{23_Arnold}, the authors consider the two-dimensional vector Laplacian using a mixed formulation with mixed variable $\sigma = \rot \, \velocity$. They use Raviart--Thomas elements of degree $k$ for the approximation of the vector field $\velocity$ and continuous Lagrange elements of the same degree for $\sigma$. Numerical experiments showed that the method achieves optimal convergence rates in the $L^2$-norm: $k$ for $\velocity$ and $k+1$ for $\sigma$ under electric and magnetic boundary conditions. However, under Dirichlet boundary conditions, the convergence rate for $\sigma$ deteriorates significantly, reducing to $k-1/2$. This loss of convergence order emerges from the fact that the differential operators do not preserve the inclusion of solution spaces in the De Rham complex, resulting in a loss of coercivity of the Laplacian operator and affecting the approximation of the auxiliary variable. Later, in 2023, Awanou, Fabien, Guzmán, and Stern~\cite{14_JGuzman} introduced a hybrid conforming method within the FEEC framework. In the 2D case, the error analysis relies on~\cite{42_Arnold}. It employs Brezzi--Douglas--Marini elements of order $k$ or Raviart--Thomas elements of order $k$ for the vector field, and continuous Lagrange elements of degree $k$ for the auxiliary variable, under electric or magnetic boundary conditions. They obtain optimal convergence for all variables.

\subsubsection*{Non-conforming methods}
On the other hand, in 2008, Brenner, Cui, Li, and Sung~\cite{4_Brenner} introduced a nonconforming finite element method for two-dimensional $\curl$–$\rot$ and $\grad$– $\divergence$ problems. The method employs weakly continuous piecewise linear elements and includes consistency terms to control the jumps in both tangential and normal components across element interfaces. It achieves a convergence rate of $2 - \varepsilon$ in the $L^2$-norm, where $\varepsilon$ depends on the singularities present in non-convex domains. The authors propose mesh refinements near reentrant corners to mitigate the effects of these singularities. Then, between 2022 and 2024, Barker, Cao, and Stern~\cite{6_Barker,7_Barker} proposed a primal hybrid nonconforming method as an extension of the original approach by Brenner et al.~\cite{4_Brenner} for higher degrees. This method avoids explicit auxiliary variables and instead uses numerical traces and Lagrange multipliers to weakly enforce continuity across elements. It employs discontinuous Galerkin elements of degree $2k-1$ for the vector field $ \velocity$, polynomials of the same degree for the numerical traces, and degree $k-1$ polynomials for the Lagrange multipliers. This setup achieves convergence rates of $ k + 1 $  for \( \velocity \) under electric boundary conditions.
 The method supports static condensation, but its accuracy and stability are sensitive to the choice of penalty parameters. 

\subsubsection*{Other mixed methods} The recent article~\cite{18_HWang} introduced a $\rot$--$\divergence$ mixed method that reformulates the two-dimensional vector Laplacian as a first-order system involving the rotational and divergence as auxiliary variables. The method uses Brezzi--Douglas--Marini elements of degree $k-1$ for $\velocity$ and continuous Lagrange elements of degree $k$ for the auxiliary variables. The error analysis is carried out only for Dirichlet boundary conditions, where the method achieves convergence rates: $k$ for $\velocity$ and $k-1/2$ for the auxiliary variables. Moreover, numerical experiments under electric and magnetic boundary conditions show optimal convergence for all variables.

\subsubsection*{HDG methods} The introduction of the rotational operator in 2D (or $\curl$ in 3D) as an auxiliary variable is usual in mixed formulations involving the vector Laplacian. The error analysis of HDG methods for these formulations follows a common structure, which requires the introduction of suitable projection operators for both the vector field and the auxiliary variables. For example, Cui and Cockburn~\cite{36_Cui} proposed in 2012 an HDG method for the vorticity–velocity–pressure formulation of the Stokes problem, using polynomials of degree $k$ for all variables. Their analysis shows optimal convergence of order $k+1$ for velocity and suboptimal convergence of order $k+1/2$ for vorticity and pressure, based on a combination of the HDG projection from \cite{8_Cockburn} and the standard $L^2$-projection. Similarly, in the context of the eddy current problem in heterogeneous media, Bustinza and Osorio~\cite{31_Osorio} proposed in 2018 an HDG scheme that incorporates a Lagrange multiplier to enforce the divergence-free condition. Using polynomials of degree $k$ for all variables, they proved convergence of order $k+1$ for the electromagnetic vector field and $k+1/2$ for auxiliary variables, with an analysis based on projection operators similar to those in \cite{36_Cui}. 

Continuing in a similar vein, Nguyen, Peraire, and Cockburn~\cite{41_Cockburn} introduced in 2011 two HDG schemes for the $\curl$-formulation of Maxwell equations, one of which incorporates a Lagrange multiplier to enforce the divergence constraint. Both methods exhibit numerically optimal convergence of order $k+1$ for the electric field and the $\curl$ variable, although theoretical guarantees for this convergence remain open when using polynomials of degree $k$ for all variables. In 2020, Du and Sayas~\cite{37_Shukai} presented a unified framework for static Maxwell equations, covering four HDG methods. Their formulation employs polynomial spaces of degree \( k + 1 \) for the vector field \( \velocity \) and degree \( k \) for the auxiliary variable \( \curl\, \velocity \), while distinguishing between different choices for the polynomial spaces used for the numerical traces and for the Lagrange multiplier \( p \), which is introduced to better control of \( \divergence \velocity \). All HDG variants they analyzed achieve optimal convergence of order \( k + 1 \) in the \( L^2 \)-norm for the auxiliary variable, and convergence of order \( k + 2 \) for the vector field. We obtain these results thanks to the careful design of stabilization terms and projection operators that align with the HDG structure.

Among recent variants of HDG methods, Chen, Qiu, Shi, and Solano~\cite{9_MSolano} proposed in 2017 a superconvergent HDG scheme for the $\curl-$formulation for steady-state Maxwell equations. By using polynomials of degree $k+1$ to approximate the vector field and degree $k$ for the auxiliary variables, along with appropriate stabilization parameters, they proved a convergence order of $k+2$ for the vector field without post-processing. Their analysis combines $L^2$-projections with a BDM-type projection. 
These recent advances in HDG approaches motivate the current development of methods for the vector Laplacian along this line of research.

\subsection*{The new HDG methods} In this paper, we derive HDG methods for a mixed formulation of the vector Laplacian \eqref{eq:vector-laplacian} using  the rotational and the divergence, namely,
\begin{subequations}\label{eq:MixedProblem}
\begin{alignat}{4}\label{eq:MixedProblem:a}
\sigma -\rot\,\velocity &=0, \quad & \text{ in } \Omega, \\ \label{eq:MixedProblem5::b}
\phi+ \divergence \velocity &=0, \quad & \text{ in } \Omega,  \\  \label{eq:MixedProblem:c}
\curl\,\sigma+\grad\,\phi &=\vect{f}, \quad & \text{ in } \Omega. 
\end{alignat}
\end{subequations}
The mixed formulation above differs from others, such as those in \cite{42_Arnold,23_Arnold,14_JGuzman}, because of the incorporation of the rotational $\sigma$ and the divergence $\phi$ of the field $\velocity$. This reformulation reduces the second-order equation \eqref{eq:vector-laplacian} to a fully first-order system \eqref{eq:MixedProblem}. 

Furthermore, \eqref{eq:MixedProblem} is complemented by the boundary conditions detailed in Table \ref{tab:table:boundarycondition:mixed_problem}, which correspond to those specified for the primal formulation in Table \ref{tab:boundarycondition}.
\begin{table}[h]
\centering
\small
\caption{Boundary conditions on $\Gamma$ for the mixed formulation \eqref{eq:MixedProblem}.}
\label{tab:table:boundarycondition:mixed_problem}
\renewcommand{\arraystretch}{1.2}
\begin{tabular}{@{}p{3cm}p{2.8cm}p{1.6cm}@{}}
\toprule
\hspace{0.2cm}Electric & Magnetic & Dirichlet \\
\midrule
\hspace{0.2cm} $\phi=0$           & $\sigma=0$         & \multirow{2}{1.8cm}{$\velocity=0$} \\[0.7ex]
\hspace{0.2cm} $\velocity \cdot \normal^\perp=0$  & $\velocity \cdot \normal=0$   & \\
\bottomrule
\end{tabular}
\end{table}

To the best of our knowledge, HDG methods for the two-dimensional vector Laplacian have not been previously studied. In this paper, we propose and analyze HDG methods using
polynomials of degree $k$ for all variables. 
We show that the methods can be implemented in three different forms of static condensation, which reduces the number of globally-coupled degrees of freedom. One of these formulations is particularly useful, as it is associated with a quadratic minimization problem.

To better understand the characteristics of the proposed HDG methods, we compare it with existing approaches. In particular, we contrast the different methods in terms of element type, approximation spaces, and order of convergence. We begin with the type of elements used, summarized in the first column of Table~\ref{tab:comparison}. The empty entries in the table correspond to quantities that cannot be compared. The methods in~\cite{23_Arnold,18_HWang} are formulated on triangular meshes, while~\cite{6_Barker,7_Barker} consider polygonal meshes. Our HDG method is also implemented on polygonal meshes. Subsequently, we examine the approximation spaces corresponding to each method, which are summarized in the remaining columns of Table~\ref{tab:comparison}.
\begin{table}[h]
\centering
\small
\renewcommand{\arraystretch}{1.8}
\begin{tabular}{>{\centering\arraybackslash}p{3.5cm} >{\centering\arraybackslash}p{4cm} c c c}
\toprule
\multirow{2}{*}{Method} 
& \multirow{2}{*}{Element} 
& \multicolumn{3}{c}{Spaces} \\
& & $\velocity$ & $\sigma$ & $\phi$ \\
\midrule
Arnold et al.~\cite{23_Arnold} & triangles & $\operatorname{RT}_k$ & $\operatorname{CG}_k$ & $-$ \\
Wang et al.~\cite{18_HWang} & triangles & $\operatorname{BDM}_{k-1}$ & $\operatorname{CG}_k$ &  $\operatorname{CG}_k$\\
Barker et al.~\cite{6_Barker,7_Barker} & polygons & $\boldsymbol{\mathcal{P}}_{2k-1}$ & $-$ & $-$ \\
HDG &  polygons & $\boldsymbol{\mathcal{P}}_{k}$ & $\mathcal{P}_{k}$ & $\mathcal{P}_{k}$ \\
\bottomrule
\end{tabular}
\caption{Comparison of element types and approximation spaces.}
\label{tab:comparison}
\end{table}

The method of Arnold et al.~\cite{23_Arnold} uses Raviart--Thomas spaces $\operatorname{RT}_k$ for the velocity and continuous Galerkin spaces $\operatorname{CG}_k$ for the auxiliary variable $\sigma$. The formulation proposed in~\cite{18_HWang} employs $\operatorname{BDM}_{k-1}$ spaces for the velocity together with $\operatorname{CG}_k$ spaces for both $\sigma$ and $\phi$. The hybrid primal methods introduced in~\cite{6_Barker,7_Barker} approximate the velocity using local vector polynomial spaces of degree less than or equal to $2k-1$, $\boldsymbol{\mathcal{P}}_{2k-1}$, and do not introduce the auxiliary variables $\sigma$ and $\phi$. In contrast, HDG methods use polynomial spaces of degree less than or equal to $k$ for vector and scalar variables, namely $\boldsymbol{\mathcal{P}}_{k}$ for the velocity and $\mathcal{P}_{k}$ for both $\sigma$ and $\phi$.

Finally, we compare the orders of convergence, summarized in Table~\ref{tab:comparison2}.  For the HDG formulation, we prove that the velocity converges with order $k+1$, while the approximations of $\sigma$ and $\phi$ converge with order $k+1/2$. Notably, these convergence rates are independent of the type of boundary conditions considered. This behavior contrasts with the method proposed in~\cite{23_Arnold}, which exhibits a loss of convergence order for the approximation of $\sigma$ under Dirichlet boundary conditions, where the rate reduces to $k-1/2$. A similar suboptimal behavior for the variables $\sigma$ and $\phi$ is also observed in the formulation of Wang et al.~\cite{18_HWang}. By contrast, the HDG method maintains the same convergence orders across electric, magnetic, and Dirichlet boundary conditions.

\begin{table}[h]
\centering
\small
\setlength{\tabcolsep}{6pt}
\renewcommand{\arraystretch}{1.8}
\begin{tabular}{
>{\centering\arraybackslash}p{3.3cm}
c c c @{\hspace{25pt}}
c c c @{\hspace{25pt}}
c c c}
\toprule
\multirow{2}{*}{Method} 
& \multicolumn{3}{c}{Electric} 
& \multicolumn{3}{c}{Magnetic}
& \multicolumn{3}{c}{Dirichlet} \\


& $\velocity$ & $\sigma$ & $\phi$
& $\velocity$ & $\sigma$ & $\phi$
& $\velocity$ & $\sigma$ & $\phi$ \\

\midrule

Arnold et al.~\cite{23_Arnold} & $k$ & $k+1$ & $-$
  & $k$ & $k+1$ & $-$
  & $k$ & $k-\tfrac{1}{2}$ & $-$ \\

Wang et al.~\cite{18_HWang} & $k$ & $k+1$ & $k+1$
  & $k$ & $k+1$ & $k+1$
  & $k$ & $k-\frac{1}{2}$ & $k-\tfrac{1}{2}$ \\

Barker et al.~\cite{6_Barker,7_Barker} & $-$ & $-$ & $-$
  & $k+1$ & $-$ & $-$
  & $-$ & $-$ & $-$ \\

HDG & $k+1$ & $k+\tfrac{1}{2}$ & $k+\tfrac{1}{2}$
  & $k+1$ & $k+\tfrac{1}{2}$ & $k+\tfrac{1}{2}$
  & $k+1$ & $k+\tfrac{1}{2}$ & $k+\tfrac{1}{2}$ \\

\bottomrule
\end{tabular}
\caption{Comparison of the order of convergence.}
\label{tab:comparison2}
\end{table}

\subsection*{Organization of the paper} The remainder of this paper is organized as follows. In Section \ref{section:TheHDGMethod}, we introduce the HDG method, present three different hybridizations, derive \textit{a priori} error estimates in the $L^2$--norm, and discuss the characterization of the HDG solution. Section \ref{sec:Proofsofthemainresults} provides the proofs for these theoretical results. In Section \ref{section:NumericalExperiments}, we validate our findings through numerical experiments. Finally, Section \ref{section:conclussions_and_ongoing_work} contains our conclusions and outlines directions for future research.


\section{Main results}\label{section:TheHDGMethod}
In this section, we introduce HDG methods for approximating the solution of the vector Laplacian problem written in the rotational–-divergence mixed formulation \eqref{eq:MixedProblem}. We approximate the vector field, its rotational and divergence, as well as the traces on the skeleton of the tangential and normal components of the vector field. We begin by introducing some standard notation, and then in Section \ref{ss:hdgmethodsandhybridizations}, we present the global numerical scheme and three types of hybridizations. In Section \ref{subsec:MainResults:Apriori_error_estimates}, we present the \textit{a  priori} error estimates and in Section \ref{subsection:Main Results:Hybridization_of_HDG_methods} we present the well-posedness of the local problems and characterization of the HDG solution.

\subsection{Notation}\label{section:TheHDGMethod:subsection:Notation} 
We begin this section by introducing standard notation for functional spaces and differential operators
\begin{alignat*}{4}
H(\operatorname{rot} )&:=\{\zb\in \Lb^2(\Omega): \rot \, \zb \in L^2(\Omega) \},& \quad H(\operatorname{div})&:=\{\zb\in \Lb^2(\Omega): \divergence \zb \in L^2(\Omega) \},\\
\mathring{H}(\operatorname{rot})&:=\{\zb\in H(\operatorname{rot}): \zb\cdot \normal^\perp=0\},
& \quad \mathring{H}(\operatorname{div})&:=\{\zb\in H(\operatorname{div}): \zb \cdot \normal=0 \},
\end{alignat*}
where $L^2(\Omega)$ is the standard space of the square integrable function in $\Omega$ and $\Lb^2(\Omega)$ denotes the space of square integrable vector functions. The operator $()^\perp$ is defined for a two-dimensional vector field as its counter-clockwise rotation, i.e., $(a,b)^\perp:=(b,-a)$. We also define for smooth scalar and vector functions $\varphi$  and $\vect{z}=(z_1,z_2)$,
$$
\begin{aligned}
\grad \varphi:=(\frac{\partial \varphi}{\partial x}, \frac{\partial \varphi}{\partial y}), \quad \curl \,\varphi & :=(\frac{\partial \varphi}{\partial y},-\frac{\partial \varphi}{\partial x}), \\ \divergence \vect{z}:=\frac{\partial z_{1}}{\partial x}+\frac{\partial z_{2}}{\partial y}, \quad
\operatorname{rot} \vect{z} & :=\frac{\partial z_{2}}{\partial x}-\frac{\partial z_{1}}{\partial y}.
\end{aligned}
$$

Let $\Th$ be a quasi-uniform mesh \cite[Definition 4.4.13]{MR2373954} consisting of disjoint elements that partition $\overline{\Omega}$, with mesh size parameter $h$, the maximum inner diameter among all elements in $\Th$. Let $\Faces$ denote the set of all edges, with $\Facesinterior$ and $\Facesboundary$ representing the subsets of the interior and boundary edges, respectively. For an element $K \in \Th$, we denote by $\partial K$ the set of edges that form its boundary, and then $\Skeleton$ denotes the union of all the boundaries of the elements in the mesh $\Th$, and is commonly called the mesh skeleton.

Let $\Kprod{\cdot}{\cdot}$ and $\pKprod{\cdot}{\cdot}$ be the $L^2$--inner products over $K$ and $\partial K$, respectively.    
In addition, the piecewise $L^2$--inner products are defined as
\[
\Thprod{\cdot}{\cdot}:=\sum_{K\in \Th} \Kprod{\cdot}{\cdot}, \quad
\pThprod{\cdot}{\cdot} :=\sum_{K\in \Th}\pKprod{\cdot}{\cdot}.
\]
Analogous definitions apply to vector-valued functions. The norms induced by the inner products above are denoted by $\Thnorm{\cdot}$ and $\pThnorm{\cdot}$.

Define the jumps of the scalar and vector function $\varphi$ and $\vect{z}$, in the normal and tangential directions across the interior edge $\partial K^{-}\cap \partial K^{+}=:F \in \Facesinterior$, for $K^{-}, K^{+}\in \Th$, by
\begin{alignat*}{4}
   \jump{\varphi \, \normal}&:=\varphi^{+}  \normal^{+}+\varphi^{-}  \normal^{-},  & \quad   \jump{\varphi \, \normal^\perp}&:=\varphi^{+} (\normal^\perp )^+ +\varphi^{-} (\normal^\perp)^{-},\\   
   \jump{\zb \cdot \normal}&:=\zb^{+} \cdot \normal^{+}+\zb^{-} \cdot \normal^{-}, & \quad   \jump{\zb \cdot \normal^\perp}&:=\zb^{+}\cdot (\normal^\perp )^+ +\zb^{-}\cdot (\normal^\perp)^{-},
\end{alignat*}
where $\normal^{\pm}$ are the outward unit normals to $K^{\pm}$, and  $\varphi^{\pm}=\varphi|_{K^{\pm}}$ and $\zb^{\pm}=\zb|_{K^{\pm}}$. On boundary edges $F\in \Facesboundary$, we simply define
\begin{alignat*}{4}
 \jump{\varphi \, \normal}&:=\varphi \, \normal,  & \quad   \jump{\varphi \, \normal^\perp}&:=\varphi \, \normal^\perp,&\quad
   \jump{\zb \cdot \normal}&:=\zb\cdot \normal, & \quad   \jump{\zb \cdot \normal^\perp}&:=\zb \cdot \normal^\perp.
\end{alignat*}
We say that the $\varphi$ is single-valued if $\jump{\varphi \, \normal}=0$ and
$\jump{\varphi \, \normal^\perp}=0$ on $\Facesinterior$. Similarly, $\zb$ is single-valued if $\jump{\zb \cdot \normal}=0$ and
$\jump{\zb \cdot \normal^\perp}=0$ on $\Facesinterior$.

Finally, we introduce the following discontinuous finite-dimensional approximation spaces. First, we introduce $\mathcal{P}_{k}(A)$  and $\boldsymbol{\mathcal{P}}_{k}(A)$  as the scalar and vector polynomial spaces of degree less than or equal to $k\geq 0$ respectively, defined over the domain $A$. Then, we define the following scalar and vector volumetric approximation function spaces
\begin{eqnarray*}
{W}_{h}&:=&\left\{ \varphi \in L^{2}(\Omega): \;\varphi |_{K} \in \PK, \;\forall K \in \mathcal{T}_{h}\right\},\\
\vect{V}_{h}&:=&\left\{ \zb \in \Lb^{2}(\Omega):\;\zb|_{K} \in  \PKb, \;\forall K \in \mathcal{T}_{h}\right\},
\end{eqnarray*}
and the approximation spaces for the numerical traces 
\begin{eqnarray*}
M_{h}&:=&\left\{ \mu \in L^{2}(\Faces) :\mu|_{F} \in \mathcal{P}_k(F)\; \text{on } \Facesinterior \text{ and }  \mu|_{F}=0\; \text{for }F\in \Facesboundary \right\},\\
\vect{M}_h&:=&  \{ \mub \in \mathbf{L}^2(\partial \mathcal{T}_h) \, : \, \mub|_{F} \in \boldsymbol{\mathcal{P}}_k(F) \;\; \forall F\in \partial K,\; K \in \mathcal{T}_h \}, \\
\Nhperp&:=& \{ \mub\in \vect{M}_h:   \mub \cdot\normal=0  \text{ on $\Faces$} \text{ and } (\jump{\mub \cdot \normal^\perp}=0 \text{ on $\Facesinterior$} \; \wedge \;\mub \cdot \normal^\perp=0 \text{ on $\Facesboundary$} )\},\\
 \Nh&:= &\{ \mub\in \vect{M}_h: \mub \cdot\normal^\perp=0  \text{ on $\Faces$} \text{ and }( \jump{\mub \cdot \normal}=0 \text{ on $\Facesinterior$}\;\wedge\; \mub \cdot \normal=0 \text{ on $\Facesboundary$} )\}. 
\end{eqnarray*} 
\subsection{The HDG method and its hybridizations}\label{ss:hdgmethodsandhybridizations}
Next, we derive the HDG method and its three types of hybridization. The HDG formulation consists of two main components: a local problem and a global problem. The first seeks approximations of the solution within each element $K$ of the mesh $\Th$, given boundary conditions imposed through functions called numerical traces-- interpreted as data for the local problems. The global problem characterizes these traces by enforcing weak continuity conditions at element edges. The three hybridization types arise from interchanging the roles of the numerical traces as global unknowns, according to the type of boundary condition imposed on the local problem.

To obtain the HDG method, we multiply the rotational-velocity mixed formulation \eqref{eq:MixedProblem} by appropriate test functions. By applying integration by parts, we arrive at the following system that defines an approximation 
\((\sigma_h, \phi_h, \velocity_h) \in \Wh \times \Wh \times \Vh\)
as the solution of
\begin{subequations}\label{eq:MixedFormulation:Weakform}
\begin{alignat}{4} \label{eq:MixedFormulation:Weakform:a}
\Thprod{\sigma_h}{\chi}-\Thprod{\velocity_h}{\curl \, \chi} + \pThprod{\velocitycheck_h\cdot \normal^\perp}{\chi}& =0,    \\  \label{eq:MixedFormulation:Weakform:b}
\Thprod{\phi_h}{\varphi}-\Thprod{\velocity_h}{\grad\varphi}+\pThprod{\velocityhat_h\cdot \normal}{\varphi
} & =0, 
\\  \label{eq:MixedFormulation:Weakform:c}
\Thprod{\sigma_h}{\rot \, \zb} +
\pThprod{\sigmacheck_h}{\zb \cdot \normal^\perp} - 
\Thprod{\phi_h}{\divergence \zb} +
\pThprod{\phihat_h}{ \zb \cdot \normal } 
& =\Thprod{\fb}{ \zb},
\end{alignat}
for all $(\chi,\varphi,\zb) \in  \Wh \times \Wh\times \Vh$. The HDG numerical traces introduced above $\widecheck{\sigma}_h$, $\widehat{\phi}_h$, $\widecheck{\vect{u}}_h$ and $\widehat{\vect{u}}_h$, are approximations of the solution at the skeleton of the triangulation and satisfy 
\begin{alignat}{4} \label{eq:MixedFormulation:Weakform:d}
     (\velocitycheck_h-\velocity_h)\cdot \normal^\perp & = \alpha(\sigma_h-\sigmacheck_h) &&\qquad \text{on } \partial \mathcal T_h,\\ \label{eq:MixedFormulation:Weakform:e}
    (\velocityhat_h-\velocity_h)\cdot \normal & =\tau(\phi_h-\phihat_h) &&\qquad \text{on } \partial \mathcal T_h,
\end{alignat}
where the stabilization parameters $\alpha > 0$ and $\tau > 0$ are constants over the edges of the mesh.

The following two equations, which impose global continuity of the numerical fluxes, complete the definition of the method
\begin{alignat}{4} \label{eq:MixedFormulation:Weakform:f}
    \pThprod{\sigmacheck_h}{\mub^t \cdot \normal^\perp}=0,& \quad \forall \mub^t \in \Nhperp,\\ \label{eq:MixedFormulation:Weakform:g}
    \pThprod{\phihat_h}{\mub^n  \cdot \normal}=0,& \quad \forall \mub^n \in \Nh.    
\end{alignat}
\end{subequations}
In addition, the HDG method enforces the boundary conditions as indicated in Table~\ref{tab:BCfortraces}.
\begin{table}[h]
\centering
\small
\caption{Boundary conditions on $\Gamma$ for \eqref{eq:MixedFormulation:Weakform}.}
\renewcommand{\arraystretch}{1.2}
\begin{tabular}{@{}p{3cm}p{2.8cm}p{2cm}@{}}
\toprule
\hspace{0.2cm}Electric & Magnetic & Dirichlet \\
\midrule
\hspace{0.2cm} $\phihat_h=0$           & $\sigmacheck_h=0$         & $\velocitycheck_h\cdot \normal^\perp=0$  \\
\hspace{0.2cm} $\velocitycheck_h \cdot \normal^\perp=0$  & $\velocityhat_h \cdot \normal=0$   &  $\velocityhat_h\cdot \normal=0$ \\
\bottomrule
\end{tabular} \label{tab:BCfortraces}
\end{table}

We can write the global equations \eqref{eq:MixedFormulation:Weakform:f} and \eqref{eq:MixedFormulation:Weakform:g} in alternative forms depending on the choice of trace unknowns.  Specifically, we introduce three different HDG hybridizations based on distinct selections of trace variables, as summarized in the second column of Table~\ref{tab:KindsHDGMethods}.  These choices lead to different strategies for enforcing the transmission conditions that guarantee continuity of the numerical fluxes across element interfaces. The trace unknowns are determined through the jump conditions listed in the third column of the same table.  

In Hybridization of Type I, the global unknowns are $\velocitycheck_h$ and $\phihat_h$; Hybridization of Type II uses $\sigmacheck_h$ and $\velocityhat_h$; and Hybridization of Type III selects $\velocitycheck_h$ and $\velocityhat_h$. Although each formulation involves a different combination of trace unknowns, all of them are equivalent and lead to global systems after hybridization.  The choice of trace variables determines which weak continuity conditions are imposed across the mesh skeleton and reflects the physical constraints that the formulation aims to preserve.
\begin{table}[h]
\centering
\small
\caption{Unknowns and jump conditions for each hybridization.}
\label{tab:KindsHDGMethods}
\begin{tabular}{@{}p{6cm}p{3cm}p{3cm}@{}}
\toprule
\multicolumn{2}{r}{Trace unknowns \phantom{ooooo}} & Jump conditions \\
\midrule
\multirow{2}{*}{\hspace{0.2cm}Hybridization of Type I} 
& $\velocitycheck_h \in \Nhperp$ 
& $\jump{\sigmacheck_h \, \normal^\perp} = 0$ \\
& $\phihat_h \in M_h$ 
& $\jump{\velocityhat_h \cdot \normal} = 0$ \\
\addlinespace[2ex]
\multirow{2}{*}{\hspace{0.2cm}Hybridization of Type II} 
& $\sigmacheck_h \in M_h$ 
& $\jump{\velocitycheck_h \cdot \normal^\perp} = 0$ \\
& $\velocityhat_h \in \Nh$ 
& $\jump{\phihat_h \, \normal} = 0$ \\
\addlinespace[2ex]
\multirow{2}{*}{\hspace{0.2cm}Hybridization of Type III} 
& $\velocitycheck_h \in \Nhperp$ 
& $\jump{\sigmacheck_h \, \normal^\perp} = 0$ \\
& $\velocityhat_h \in \Nh$ 
& $\jump{\phihat_h \, \normal} = 0$ \\
\bottomrule
\end{tabular}
\end{table}

Global equations \eqref{eq:MixedFormulation:Weakform:f}--\eqref{eq:MixedFormulation:Weakform:g} correspond to the Hybridization of Type III (see the third row of Table~\ref{tab:KindsHDGMethods}), where these equations arise from the associated jump conditions.  The global equations for Hybridization of Type I and Hybridization of Type II can be derived analogously from the jump conditions listed in the first and second rows of Table~\ref{tab:KindsHDGMethods}, respectively. 

Without loss of generality, and unless stated otherwise, we conduct all subsequent analysis, proofs, and results under Dirichlet boundary conditions using the Hybridization of Type III as the reference formulation.
\subsection{Main Results: \textit{A priori} error estimates}\label{subsec:MainResults:Apriori_error_estimates}
In this section, we present the convergence and accuracy results of the scheme \eqref{eq:MixedFormulation:Weakform} in the $L^2$-norm. First, we introduce the estimates for the approximations of the auxiliary variables, the divergence $\phi$ and the rotational $\sigma$ and their respective numerical fluxes, and then present the estimate for the approximation of the vector field $\vect{u}$ using a duality argument. We use this argument to obtain an improved convergence order for the approximation of $\velocity$, see related works \cite{1P_Monk,8_Cockburn,18_HWang}. To this end, we introduce the dual problem: for any $\Theta \in \Lb^2(\Omega)$, we seek $\eta$, $\Phi$, and $\vect{\Psi}$ such that 
\begin{subequations}\label{eq:dual-problem}
 \begin{alignat}{4} \label{eq:dual-problem:a}
\eta-\rot \,\vect{\Psi}&=0,&\quad \text{ in }\Omega, \\ \label{eq:dual-problem:b}
\Phi + \divergence\,\vect{\Psi}&=0, &\quad \text{ in }\Omega,\\ \label{eq:dual-problem:c}
\curl\,\eta+\grad\,\Phi &=\Theta,& \quad \text{ in }\Omega,\\
\label{eq:dual-problem:d}
\vect{\Psi}  &=0,& \quad \text{ on }\Gamma,
\end{alignat}
and we also assume that the solution admits the following regularity estimate
  \end{subequations}
  \begin{eqnarray}\label{eq:dual-problem:estimate}
 \hspace{2cm}     \| \Phi \|_{H^1(\Omega)}+      \| \vect{\Psi} \|_{\vect{H}^2(\Omega)}+ \| \eta \|_{H^1(\Omega)}\leq C \| \Theta \|_{L^2(\Omega)}, 
  \end{eqnarray}
for $C>0.$ The above estimate follows from~\cite{17_BKellogg} for convex domains in two dimensions. We are now ready to state our first main result.
\begin{theorem} \label{Teo:A_priori_error_estimates} (\textit{A priori} error estimates) Assume that the solution of \eqref{eq:MixedProblem} satisfies  $\sigma \in H^{t}(\Omega)$, $\phi \in H^{r}(\Omega)$, and $\velocity \in \vect{H}^{s}(\Omega)$, 
and that $\sigma_h, \phi_h \in \Wh$, and $\velocity_h \in \Vh$  are the solutions of \eqref{eq:MixedFormulation:Weakform},  then
\begin{subequations}\label{eq:Theorem:A_priori_error_estimate}
\begin{alignat}{4} \label{eq:Theorem:A_priori_error_estimate:a}
    \Thnorm{\sigma-\sigma_h}+    \Thnorm{\phi-\phi_h}&\leq C ( h^{s-\frac{1}{2}} \Hs{\velocity} +h^{t-\frac{1}{2}} \Ht{\sigma} + h^{r-\frac{1}{2}} \Hr{\phi}  ),\\ \label{eq:Theorem:A_priori_error_estimate:a2}
 \pThnorm{\sigma-\sigmacheck_h}+  \pThnorm{\phi-\phihat_h}&\leq C(h^{s-1}\Hs{\velocity}+h^{t-1}\Ht{\sigma} +h^{r-1}\Hr{\phi}),     
    \end{alignat}
for $1\leq s,t,r \leq k+1$.\\ 
Furthermore, if the regularity estimate \eqref{eq:dual-problem:estimate} is verified for the solution of the dual problem \eqref{eq:dual-problem}, then 
    \begin{alignat}{4} \label{eq:Theorem:A_priori_error_estimate:b}
        \Thnorm{\velocity-\velocity_h}&\leq C ( h^{s} \Hs{\velocity} +h^{t} \Ht{\sigma} + h^{r} \Hr{\phi}  ).
\end{alignat}

\end{subequations}
\end{theorem}
We provide a proof of this theorem in Section~\ref{section:Proofs_WellPosedness_A_priori_error_estimates}, considering Dirichlet boundary conditions. Once the estimates stated above are established, the unisolvency of the HDG method follows as a direct consequence. Indeed, since the system~\eqref{eq:MixedFormulation:Weakform} is a square linear system, setting $\fb = 0$ implies that the exact solution vanishes. Then, by the estimates \eqref{eq:Theorem:A_priori_error_estimate}, the HDG solution must also vanish.

Theorem~\ref{Teo:A_priori_error_estimates} further shows that the approximation of the vector field $\velocity_h$ has optimal convergence of order $k+1$, while the auxiliary variables $\sigma_h$ and $\phi_h$ converge at a rate of $k+1/2$, which is suboptimal by $1/2$. This loss in accuracy is due to the use of the standard $L^2$-projection in the analysis, which, although applicable to general quasi-uniform meshes, is reduced by $1/2$ due to trace term estimates. As a result, the convergence rates of the auxiliary variables $\sigma_h$ and $\phi_h$ are reduced.

In \cite{23_Arnold}, the authors consider a mixed formulation involving the rotational variable $\sigma = \rot \, \velocity$ under Dirichlet boundary conditions. They prove that the use of conforming finite element spaces, specifically Raviart--Thomas elements of degree $k$ for the vector field and continuous Lagrange elements of the same degree for the auxiliary variable, does not preserve the inclusion of solution spaces under differential operators in the De Rham complex when Dirichlet boundary conditions are imposed. Consequently, the approximation $\sigma_h$ has suboptimal convergence of order $k - 1/2$, while the vector field $\velocity_h$ converges optimally with order $k + 1$. Although the formulations are different, our error analysis provides improved convergence rates for the auxiliary variables compared to those achieved with this conforming finite element approach.
\subsection{Main Results: Hybridizations of HDG methods}\label{subsection:Main Results:Hybridization_of_HDG_methods}
This section shows that the three hybridizations are well-defined, that is, treating $\fb$ and the trace unknowns as given data for local problems ensures that these problems have a unique solution. We also provide a characterization of the HDG solutions,   which illustrates the structure of the methods and how static condensation can be applied. Proofs of the results in this section are provided in Section~\ref{sec:Proofsofthemainresults}.

\begin{theorem}(Well-posedness  of the local problem) \label{Teo:Localproblem}  
Let $\fb$ and the boundary conditions $\velocity_h \cdot \normal^\perp=\velocitycheck_h \cdot \normal^\perp$ and $\velocity\cdot \normal = \velocityhat_h \cdot \normal$ on $\partial K$. Then, the local problem \eqref{eq:MixedFormulation:Weakform:a}--\eqref{eq:MixedFormulation:Weakform:c} admits a unique solution.
\end{theorem}
 For the proof of this theorem, see Section \ref{section:Proofs_Hybridization}.
 
Next, we characterize the HDG solutions by splitting them based on trace unknowns and $\fb$. Additionally, we derive the energy functional associated with the corresponding global problem.  For clarity of presentation, we focus on the Hybridization of Type III case in what follows. The results for the other cases (Hybridization of Type I and Hybridization of Type II) are provided in the Appendix~\ref{sec:Appendix}.
\begin{theorem} (Characterization of HDG solution) \label{Teo:CharacterizationHDGsolution} 
The approximate solution defined in \eqref{eq:MixedFormulation:Weakform} using the Hybridization of Type III is characterized by 
\begin{subequations} \label{eq:Characterization-of-error:theorem3firstproblem}
\begin{alignat*}{4}
(\sigma_h,\phi_h,\velocity_h)=(\Sigma,\Phi,\vect{U})=(\Sigmafourth,\Phifourth,\Ufourth)+(\Sigmafourthf,\Phifourthf,\Ufourthf),  
\end{alignat*}
where, on the element $K\in \Th,$ for any $\etab^t\in \Lb^2(\partial K),$ $\etab^n\in \Lb^2(\partial K),$ the function \\ $(\Sigmafourtheta,\Phifourtheta,\Ufourtheta)$ in $\PK\times \PK \times \PKb$ is the solution of the  first local problem
\begin{alignat}{4} \label{eq:Characterization-of-error:theorem3:a}
\Kprod{\Sigmafourtheta}{\chi}-\Kprod{\Ufourtheta}{\curl \, \chi}+ \pKprod{\etab^t\cdot \normal^\perp}{\chi
} & =0,   \quad \forall \chi \in \PK,\\  \label{eq:Characterization-of-error:theorem3:b}
\Kprod{\Phifourtheta}{\varphi}-\Kprod{\Ufourtheta}{\grad\varphi}+\pKprod{\etab^n \cdot \normal}{\varphi
} & =0, \quad \forall \varphi \in \PK, \\ 
\Kprod{\Sigmafourtheta}{\rot \, \zb}+\pKprod{\Sigmacheckfourths }{\zb\cdot \normal^\perp
}   - 
\Kprod{\Phifourtheta}{\divergence \zb}
+\pKprod{\Phihatfourths}{\zb \cdot \normal}
& = 0, \quad \forall \zb \in \PKb,\label{eq:Characterization-of-error:theorem3:c}
\end{alignat}
where 
\begin{alignat}{4}  \label{eq:Characterization-of-error:theorem3:d}
\Sigmacheckfourths&=\Sigmafourtheta+\alpha^{-1} (\Ufourtheta  -\etab^t)\cdot \normal^\perp, & \text{ on }\partial K,\\ \label{eq:Characterization-of-error:theorem3:e}
\Phihatfourths&=\Phifourtheta+\tau^{-1} (\Ufourtheta  -\etab^n)\cdot \normal, & \text{ on }\partial K.
\end{alignat}
\end{subequations}
\begin{subequations} \label{eq:Characterization-of-error:theorem3secondproblem}
We also have that any $\fb\in \Lb^2(K),$ then $(\Sigmafourthf,\Phifourthf,\Ufourthf)$ in $\PK\times \PK \times \PKb$ is the solution of the second local problem  
\begin{alignat}{4} \label{eq:Characterization-of-error:theorem3:f}
\Kprod{\Sigmafourthf}{\chi}-\Kprod{\Ufourthf}{\curl \, \chi}  & =0, & \forall \chi \in \PK,\\  \label{eq:Characterization-of-error:theorem3:g}
\Kprod{\Phifourthf}{\varphi}-\Kprod{\Ufourthf}{\grad\varphi} & =0, & \forall \varphi \in \PK ,& 
\\ 
\Kprod{\Sigmafourthf}{\rot \, \zb}+ \pKprod{\widecheck{\Sigma}_\fb }{\zb \cdot  \normal^\perp} - 
\Kprod{\Phifourthf}{\divergence \zb}
+\pKprod{\widehat{\Phi}_\fb}{\zb \cdot \normal} & = \Kprod{\fb}{ \zb}, & \ \forall \zb \in \PKb,\label{eq:Characterization-of-error:theorem3:h}
\end{alignat}
where
\begin{alignat}{4}  \label{eq:Characterization-of-error:theorem3:i}
\widecheck{\Sigma}_{\fb}&=\Sigmasecondf+\alpha^{-1} \Usecondf \cdot \normal^\perp, & \text{ on }\partial K,\\ \label{eq:Characterization-of-error:theorem3:j}
\widehat{\Phi}_{\fb}&=\Phifourthf+\tau^{-1} \Ufourthf \cdot \normal, & \text{ on }\partial K.
\end{alignat}
\end{subequations} 
\begin{subequations}  \label{eq:Characterization-of-error:theorem3energy}
Here, we have  $(\velocitycheck_h ,\velocityhat_h)\in \Nhperp \times \Nh$ is the solution of
\begin{alignat}{4}  \label{eq:Characterization-of-error:theorem3:k}
a_h((\velocitycheck_h ,\velocityhat_h),(\mub^t,\mub^n))=l_h(\mub^t,\mub^n),
\end{alignat}
where
\begin{alignat}{4} \nonumber   
a_h((\etab^t ,\etab^n),(\mub^t,\mub^n))&= \Thprod{\Sigmafourtheta}{\Sigmafourthmu}+ \Thprod{\Phifourtheta}{\Phifourthmu} \\ \nonumber
&+\pThprod{\alpha^{-1} (\Ufourtheta   - \etab^t)\cdot \normal^\perp}{(\Ufourthmu   - \mub^t)\cdot \normal^\perp}  \\
&+\pThprod{\tau^{-1} (\Ufourtheta   - \etab^n)\cdot \normal}{(\Ufourths   - \mub^n)\cdot \normal}, 
\label{eq:Characterization-of-error:theorem3:n}\\
l_h(\mub^t,\mub^n)&= \Thprod{\fb}{\Ufourthmu},\label{eq:Characterization-of-error:theorem3:o}
\end{alignat}
and $(\velocitycheck_h,\velocityhat_h)$ minimizes the total energy functional
\end{subequations}
\begin{alignat}{4} \label{eq:Energy}
\mathcal{J}_h(\mub^t,\mub^n):=\frac{1}{2}a_h((\mub^t,\mub^n),(\mub^t,\mub^n))-l_h(\mub^t,\mub^n)
\end{alignat}
over the functions $(\mub^t,\mub^n)   \in \Nhperp \times \Nh$.
\end{theorem} 
A proof of the above theorem can be found in Section \ref{section:Proofs_Hybridization}. For the other cases, Hybridization of Type I–Hybridization of Type II, the results are provided in the Appendix~\ref{sec:Appendix}, the proofs are analogous and are therefore omitted.

\section{Proofs of the main results} \label{sec:Proofsofthemainresults}
In this section, we present the proofs of the results stated in Section~\ref{section:TheHDGMethod}, the \textit{a priori} error estimates in the $L^2$--norm (Theorem~\ref{Teo:A_priori_error_estimates}), the well-posedness of the local problem (Theorem \ref{Teo:Localproblem}), and the characterizations of the HDG solutions (Theorem \ref{Teo:CharacterizationHDGsolution}).
\subsection{Proof of Theorem \ref{Teo:A_priori_error_estimates} (\textit{A priori} error estimates) }\label{section:Proofs_WellPosedness_A_priori_error_estimates}
We begin by introducing the standard $L^2$--projection estimates and proceed as follows. In \textit{Step 1}, we derive the equations satisfied by the projection of the error. In \textit{Steps 2} and \textit{3}, we establish an energy argument, which is utilized in \textit{Step 4} to estimate the $L^2$--error for the approximations of rotational and divergence variables. Similarly, in \textit{Step 5}, we estimate the error for the numerical traces. Finally, in \textit{Step 6}, we apply a duality argument to obtain error estimates for the vector field.

First, we define the following projection errors
\begin{alignat*}{4}
\errorsigma := \PiW \sigma - \sigma_h, \ 
\errorphi := \PiW \phi - \phi_h, \
\erroru := \PiV \velocity - \velocity_h, \
\errorsigmacheck := \PM \sigma - \sigmacheck_h, \ 
\errorphihat := \PM \phi - \phihat_h,
\end{alignat*}
where $\PiW$, $\PiV$, and $\PM$ are the standard $L^2$--projections on the finite-dimensional spaces $\Wh$, $\Vb_h$, and $M_h$, respectively. In this analysis, we use the following standard approximation properties
\begin{subequations}\label{eq:Cotas_proyeccion_L2}
\begin{alignat}{4}\label{eq:Cotas_proyeccion_L2:a}
\Thnorm{\sigma-\PiW \sigma}&\leq C h^{t} \Ht{\sigma},& \quad 0\leq t \leq k+1,\\ \label{eq:Cotas_proyeccion_L2:b}
\pThnorm{\sigma-\PiW \sigma}&\leq C h^{t-\frac{1}{2}} \Ht{\sigma},& \quad 1\leq t \leq k+1,\\ \label{eq:Cotas_proyeccion_L2:c}
\pThnorm{\sigma-\PM \sigma}&\leq C h^{t-\frac{1}{2}} \Ht{\sigma},& \quad 1\leq t \leq k+2,\\                   \label{eq:Cotas_proyeccion_L2:d}
\Thnorm{\phi-\PiW \phi}&\leq C h^{r} \Hr{\phi},& \quad 0\leq r \leq k+1, \\ \label{eq:Cotas_proyeccion_L2:f}
\pThnorm{\phi-\PiW \phi}&\leq C h^{r-\frac{1}{2}} \Hr{\phi},& \quad 1\leq r \leq k+1, \\
\label{eq:Cotas_proyeccion_L2:e}
\pThnorm{\phi-\PM \phi}&\leq C h^{r-\frac{1}{2}} \Hr{\phi},& \quad 1\leq r \leq k+2, \\ \label{eq:Cotas_proyeccion_L2:g}
\Thnorm{\velocity-\PiV \velocity}&\leq C h^s \Hs{\velocity}, & \quad 0\leq s \leq k+1,\\ \label{eq:Cotas_proyeccion_L2:h}
\pThnorm{\velocity-\PiV \velocity}&\leq C h^{s-\frac{1}{2}} \Hs{\velocity},& \quad 1\leq s \leq k+1.
\end{alignat}
\end{subequations}
\subsubsection*{Step 1: The error equations.} We begin with some auxiliary results that characterize the error of the HDG method.
\begin{lemma}\label{lemma:ecuaciones_del_error} Let $(\sigma,\phi,\velocity)$ and $(\sigma_h,\phi_h,\velocity_h,\velocitycheck_h,\velocityhat_h)$ solutions of the systems of equations \eqref{eq:MixedProblem} and \eqref{eq:MixedFormulation:Weakform}, respectively. Then, we have
\begin{subequations}\label{eq:lemma1:equacion_error}
\begin{alignat}{4} \label{eq:lemma1:equacion_error:a}
\Thprod{\errorsigma}{\chi}-\Thprod{\erroru}{\curl \, \chi} + \pThprod{\velocity\cdot \normal^\perp-\velocitycheck_h\cdot \normal^\perp}{\chi}& =0,    \\  \label{eq:lemma1:equacion_error:b}
\Thprod{\errorphi}{\varphi}-\Thprod{\erroru}{\grad\varphi}+\pThprod{\velocity\cdot \normal-\velocityhat_h\cdot \normal}{\varphi
} & =0, 
\\  \label{eq:lemma1:equacion_error:c}
\Thprod{\errorsigma}{\rot \, \zb} +
\pThprod{\errorsigmacheck}{\zb \cdot \normal^\perp} - 
\Thprod{\errorphi}{\divergence \zb} +
\pThprod{\errorphihat}{ \zb \cdot \normal } 
& =0,\\ \label{eq:lemma1:equacion_error:d}
\pThprod{\errorsigmacheck}{\mub^t \cdot \normal^\perp} 
& =0,\\\label{eq:lemma1:equacion_error:e}
\pThprod{\errorphihat}{\mub^n \cdot \normal} & =0,
\end{alignat}
for all $(\chi,\varphi,\zb,\mub^t,\mub^n)\in \Wh\times \Wh\times\vect{V}_h\times \Nhperp\times \Nh.$
\end{subequations}
\end{lemma}
\begin{proof}[Proof:]
The exact solutions $\sigma,$ $\phi$ and $\velocity$ verify
\begin{subequations}
\begin{alignat*}{4} 
\Thprod{\sigma}{\chi}-\Thprod{\velocity}{\curl \, \chi} + \pThprod{\velocity \cdot \normal^\perp}{\chi}& =0,    \\  
\Thprod{\phi}{\varphi}-\Thprod{\velocity}{\grad\varphi}+\pThprod{\velocity \cdot \normal}{\varphi
} & =0, 
\\  
\Thprod{\sigma}{\rot \, \zb} +
\pThprod{\sigma}{\zb \cdot \normal^\perp} - 
\Thprod{\phi}{\divergence \zb} +
\pThprod{\phi}{ \zb \cdot \normal } 
& =\Thprod{\fb}{\zb},\\ 
\pThprod{\sigma}{\mub^t \cdot \normal^\perp} 
& =0,\\
\pThprod{\phi}{\mub^n \cdot \normal} & =0,
\end{alignat*}
\end{subequations}
for all $(\chi,\varphi,\zb,\mub^t,\mub^n)\in \Wh\times \Wh\times\vect{V}_h\times \Nhperp \times \Nh.$    
Using the definition of $L^2-$projections, we obtain
\begin{subequations}\label{prof:lemma1:equacion_error_proj}
\begin{alignat}{4} \label{proof:lemma1:equacion_error_proj:a}
\Thprod{\PiW\sigma}{\chi}-\Thprod{\PiV\velocity}{\curl \, \chi} + \pThprod{\velocity \cdot \normal^\perp}{\chi}& =0,    \\  \label{proof:lemma1:equacion_error_proj:b}
\Thprod{\PiW\phi}{\varphi}-\Thprod{\PiV\velocity}{\grad\varphi}+\pThprod{\velocity \cdot \normal}{\varphi
} & =0, 
\\  \label{proof:lemma1:equacion_error_proj:c}
\Thprod{\PiW\sigma}{\rot \, \zb} +
\pThprod{\PM\sigma}{\zb \cdot \normal^\perp} - 
\Thprod{\PiW\phi}{\divergence \zb} +
\pThprod{\PM\phi}{ \zb \cdot \normal } 
& =\Thprod{\fb}{\zb},\\ \label{proof:lemma1:equacion_error_proj:d}
\pThprod{P_M\sigma}{\mub^t \cdot \normal^\perp} 
& =0,\\\label{proof:lemma1:equacion_error_proj:e}
\pThprod{P_M\phi}{\mub^n \cdot \normal} & =0,
\end{alignat}
\end{subequations}
Subtracting system \eqref{eq:MixedFormulation:Weakform} from system \eqref{prof:lemma1:equacion_error_proj} yields the desired result. \\
\end{proof}

\subsubsection*{Step 2: Energy argument.} Next, we present an energy result associated with the previous error.
\begin{lemma}\label{lemma:energia_flujo} We have
\begin{alignat}{4}\nonumber
\Thnorm{\errorsigma}^2+\Thnorm{\errorphi}^2+\alpha \pThnorm{\errorsigma-\errorsigmacheck}^2+\tau \pThnorm{\errorphi-\errorphihat}^2=-\pThprod{(\velocity-\PiV \velocity)\cdot \normal^\perp}{\errorsigma-\errorsigmacheck}\\ \nonumber
-\pThprod{\alpha(\PM \sigma-\PiW \sigma)}{\errorsigma-\errorsigmacheck}-\pThprod{(\velocity-\PiV \velocity)\cdot \normal}{\errorphi-\errorphihat}
\\-\pThprod{\tau(\PM \phi-\PiW \phi)}{\errorphi-\errorphihat}:=-T_1-T_2-T_3-T_4.\label{eq:lemma:energia_flujo}
\end{alignat}
\end{lemma}
\begin{proof}[Proof:]
Taking $\chi=\errorsigma$ and $\varphi=\errorphi$ in equations \eqref{eq:lemma1:equacion_error:a}-\eqref{eq:lemma1:equacion_error:b}, then
\begin{subequations}\label{eq:lemma2:equacion_error}
\begin{alignat*}{4} 
\Thprod{\errorsigma}{\errorsigma}-\Thprod{\erroru}{\curl \, \errorsigma} + \pThprod{\velocity\cdot \normal^\perp-\velocitycheck_h\cdot \normal^\perp}{\errorsigma}& =0,    \\  
\Thprod{\errorphi}{\errorphi}-\Thprod{\erroru}{\grad\errorphi}+\pThprod{\velocity\cdot \normal-\velocityhat_h\cdot \normal}{\errorphi} & =0,
\end{alignat*}
and  integrating by parts, we obtain
\begin{eqnarray*}
\Thprod{\rot \,\erroru}{\errorsigma} &=& \Thprod{\errorsigma}{\errorsigma} - \pThprod{\erroru\cdot \normal^\perp}{\errorsigma}+\pThprod{(\velocity -\velocitycheck_h) \cdot \normal^\perp}{\errorsigma}, \label{eq:lemma2:equacion_error:c}  \\
\Thprod{\divergence \erroru}{\errorphi} &=&- \Thprod{\errorphi}{\errorphi}+\pThprod{\erroru \cdot \normal }{ \errorphi} - \pThprod{(\velocity-\velocityhat_h)\cdot \normal}{\errorphi}. \label{eq:lemma2:equacion_error:d}
\end{eqnarray*}
Now, replacing the last two equations in equation \eqref{eq:lemma1:equacion_error:c} with $\zb=\erroru$, we get
\begin{alignat}{4}  \nonumber
\Thprod{\errorsigma}{\errorsigma} - \pThprod{\erroru\cdot \normal^\perp}{\errorsigma}+\pThprod{(\velocity -\velocitycheck_h) \cdot \normal^\perp}{\errorsigma} +
\pThprod{\errorsigmacheck}{\erroru \cdot \normal^\perp} + \Thprod{\errorphi}{\errorphi}\\ \label{eq:lemma2:equacion_error:e}
-\pThprod{\erroru \cdot \normal }{ \errorphi} + \pThprod{(\velocity-\velocityhat_h)\cdot \normal}{\errorphi}+\pThprod{\errorphihat}{ \erroru \cdot \normal } 
 =0.
\end{alignat}
On the other hand, we have
\begin{alignat}{4} \nonumber
- \pThprod{\erroru\cdot \normal^\perp}{\errorsigma}+\pThprod{(\velocity -\velocitycheck_h) \cdot \normal^\perp}{\errorsigma} +
\pThprod{\errorsigmacheck}{\erroru \cdot \normal^\perp} = \pThprod{\erroru\cdot \normal^\perp}{\errorsigmacheck-\errorsigma}\\ \nonumber
+\pThprod{(\velocity -\velocitycheck_h) \cdot \normal^\perp}{\errorsigma}\\\nonumber
= \pThprod{\erroru\cdot \normal^\perp}{\errorsigmacheck-\errorsigma}
+\pThprod{(\velocity -\velocitycheck_h) \cdot \normal^\perp}{\errorsigma-\errorsigmacheck}+ \pThprod{(\velocity -\velocitycheck_h) \cdot \normal^\perp}{\errorsigmacheck}\\
= \pThprod{\errorsigma-\errorsigmacheck}{(\velocity -\velocitycheck_h-\erroru)\cdot \normal^\perp}, \label{eq:lemma2:equacion_error:f}
\end{alignat}
where $\pThprod{(\velocity -\velocitycheck_h) \cdot \normal^\perp}{\errorsigmacheck}=0$ because $\velocity \cdot \normal^\perp $ is single-valued and equation \eqref{eq:lemma1:equacion_error:d}. Similarly, since $\velocity \cdot \normal$ is single-valued and equation  \eqref{eq:lemma1:equacion_error:e}, then
\begin{alignat}{4} \nonumber
- \pThprod{\erroru\cdot \normal}{\errorphi}+\pThprod{(\velocity -\velocityhat_h) \cdot \normal}{\errorphi} +
\pThprod{\errorphihat}{\erroru \cdot \normal} = \pThprod{\erroru\cdot \normal}{\errorphihat-\errorphi}\\  \nonumber
+\pThprod{(\velocity -\velocityhat_h) \cdot \normal}{\errorphi}\\\nonumber
= \pThprod{\erroru\cdot \normal}{\errorphihat-\errorphi}
+\pThprod{(\velocity -\velocityhat_h) \cdot \normal}{\errorphi-\errorphihat}+ \pThprod{(\velocity -\velocityhat_h) \cdot \normal}{\errorphihat}\\
= \pThprod{\errorphi-\errorphihat}{(\velocity -\velocityhat_h-\erroru)\cdot \normal} .\label{eq:lemma2:equacion_error:g}
\end{alignat}
Replacing equations \eqref{eq:lemma2:equacion_error:f} and \eqref{eq:lemma2:equacion_error:g} in the equation \eqref{eq:lemma2:equacion_error:e} yields
\begin{alignat}{4}  \nonumber
\Thprod{\errorsigma}{\errorsigma}  + \Thprod{\errorphi}{\errorphi}+
 \pThprod{\errorsigma-\errorsigmacheck}{(\velocity -\velocitycheck_h-\erroru) \cdot \normal^\perp}\\+\pThprod{\errorphi-\errorphihat}{(\velocity -\velocityhat_h-\erroru) \cdot \normal}=0. \label{eq:lemma2:equacion_error:h}
\end{alignat}
Additionally, 
\begin{alignat}{4}  \nonumber
(\velocity -\velocitycheck_h-\erroru) \cdot \normal^\perp&=\velocity\cdot \normal^\perp -\velocitycheck_h\cdot \normal^\perp-(\PiV \velocity-\velocity_h) \cdot \normal^\perp \\ \nonumber
&=\velocity\cdot \normal^\perp -\PiV \velocity \cdot \normal^\perp+(\velocity_h-\velocitycheck_h)\cdot \normal^\perp\\ \nonumber
&=\velocity\cdot \normal^\perp -\PiV \velocity \cdot \normal^\perp-\alpha(\sigma_h-\sigmacheck_h)\\
&=\velocity\cdot \normal^\perp -\PiV \velocity \cdot \normal^\perp+\alpha(\errorsigma-\errorsigmacheck)+\alpha(\PM \sigma-\PiW\sigma),
\label{eq:lemma2:equacion_error:i}
\end{alignat}
and
\begin{alignat}{4}  \nonumber
(\velocity -\velocityhat_h-\erroru) \cdot \normal&=\velocity\cdot \normal -\velocityhat_h\cdot \normal-(\PiV \velocity-\velocity_h) \cdot \normal \\ \nonumber
&=\velocity\cdot \normal -\PiV \velocity \cdot \normal+(\velocity_h-\velocityhat_h)\cdot \normal\\ \nonumber
&=\velocity\cdot \normal -\PiV \velocity \cdot \normal-\tau(\phi_h-\phihat_h)\\
&=\velocity\cdot \normal -\PiV \velocity \cdot \normal+\tau(\errorphi-\errorphihat)+\tau(\PM \phi-\PiW\phi).
\label{eq:lemma2:equacion_error:j}
\end{alignat}
Replacing equations \eqref{eq:lemma2:equacion_error:i} and \eqref{eq:lemma2:equacion_error:j} in \eqref{eq:lemma2:equacion_error:h}, we obtain the equation \eqref{eq:lemma:energia_flujo} stated in the lemma.

\end{subequations}
\end{proof}
\subsubsection*{Step 3: Energy estimate.} We next establish the following bound for the error terms.
\begin{lemma}\label{lemma3:cota_de_energia} The errors satisfy the bound
  \begin{alignat}{4}\nonumber
\Thnorm{\errorsigma}^2+\Thnorm{\errorphi}^2+\alpha \pThnorm{\errorsigma-\errorsigmacheck}^2+\tau \pThnorm{\errorphi-\errorphihat}^2\\ \nonumber \leq Ch^{s-\frac{1}{2}}\Hs{\velocity}(\pThnorm{\errorsigma-\errorsigmacheck}+\pThnorm{\errorphi-\errorphihat})\\
+C(h^{t-\frac{1}{2}}\Ht{\sigma}\pThnorm{\errorsigma-\errorsigmacheck}+h^{r-\frac{1}{2}}\Hr{\phi}\pThnorm{\errorphi-\errorphihat} ) \label{eq:lemma3:cota_de_energia}
\end{alignat}  
where $C>0$ and $1\leq s,t, r \leq k+1.$
\end{lemma}
\begin{proof}[Proof:]
\begin{subequations} \label{eq_proof_lemma3:cota_de_energia}
In the equation \eqref{eq:lemma:energia_flujo} of Lemma \ref{lemma:energia_flujo}, we find bounds for terms $T_1,$ $T_2,$ $T_3,$ and $T_4$. By the Cauchy-Schwarz inequality and the estimate \eqref{eq:Cotas_proyeccion_L2:h}, we have
 \begin{alignat}{4} 
 \nonumber 
T_1&=\pThprod{(\velocity-\PiV \velocity)\cdot \normal^\perp}{\errorsigma-\errorsigmacheck}\leq \pThnorm{\velocity-\PiV \velocity}\pThnorm{\errorsigma-\errorsigmacheck} \label{eq_proof_lemma3:cota_de_energia:c}\\
 &\leq C  h^{s-\frac{1}{2}}\Hs{\velocity}\pThnorm{\errorsigma-\errorsigmacheck},
 \end{alignat}
with $1\leq s \leq k+1.$ Since estimates \eqref{eq:Cotas_proyeccion_L2:b}-\eqref{eq:Cotas_proyeccion_L2:c}, we follow     
 \begin{alignat}{4} 
 \nonumber
T_2&=\pThprod{\alpha(\PM \sigma-\PiW \sigma)}{\errorsigma-\errorsigmacheck}\leq \alpha \pThnorm{\PM \sigma-\PiW \sigma}  \pThnorm{\errorsigma-\errorsigmacheck}\\ \nonumber & \leq \alpha (\pThnorm{\PM \sigma- \sigma}+ \pThnorm{\sigma-\PiW\sigma})  \pThnorm{\errorsigma-\errorsigmacheck}\\
 &\leq C \alpha h^{t-\frac{1}{2}}\Ht{\sigma}\pThnorm{\errorsigma-\errorsigmacheck}, \label{eq_proof_lemma3:cota_de_energia:d}
\end{alignat}    
for $1\leq t \leq k+1.$ Now, by the estimate \eqref{eq:Cotas_proyeccion_L2:h}, then
\begin{alignat}{4}\nonumber
T_3&=\pThprod{(\velocity-\PiV \velocity)\cdot \normal}{\errorphi-\errorphihat}\leq\pThnorm{\velocity-\PiV \velocity} \pThnorm{\errorphi-\errorphihat}\\ \label{eq_proof_lemma3:cota_de_energia:a}
&\leq C h^{s-\frac{1}{2}}\Hs{\velocity}\pThnorm{\errorphi-\errorphihat},
\end{alignat} 
where $1\leq s \leq k+1.$ Finally, using bounds \eqref{eq:Cotas_proyeccion_L2:d}-\eqref{eq:Cotas_proyeccion_L2:e}, we obtain
\begin{alignat}{4}\nonumber
T_4&=\pThprod{\tau(\PM \phi-\PiW \phi)}{\errorphi-\errorphihat}\leq \tau \pThnorm{\PM \phi-\PiW \phi}  \pThnorm{\errorphi-\errorphihat}\\  & \leq \tau (\pThnorm{\PM \phi- \phi}+ \pThnorm{\phi-\PiW\phi})  \pThnorm{\errorphi-\errorphihat}\leq C \tau h^{r-\frac{1}{2}}\Hr{\phi}\pThnorm{\errorphi-\errorphihat}, \label{eq_proof_lemma3:cota_de_energia:b}
 \end{alignat}
with $1\leq r \leq k+1.$ 

The estimate \eqref{eq:lemma3:cota_de_energia}  follows from inserting \eqref{eq_proof_lemma3:cota_de_energia} into expression \eqref{eq:lemma:energia_flujo}.
\end{subequations}
\end{proof}

\subsubsection*{Step 4: Estimates of rotational and divergence.} We now prove the first part of the Theorem \ref{Teo:A_priori_error_estimates}.

\begin{proof} 
By triangle inequality, we have
\begin{alignat}{4}\nonumber
    \Thnorm{\sigma-\sigma_h}+    \Thnorm{\phi-\phi_h}&\leq    \Thnorm{\sigma-\PiW \sigma}+\Thnorm{\errorsigma}+    \Thnorm{\phi-\PiW\phi}+\Thnorm{\errorphi}\\ 
&:=\widebar{T}_1+\widebar{T}_2+\widebar{T}_3+\widebar{T}_4. \label{eq:theorem_estimates_fluxes}
    \end{alignat}
 The terms $\widebar{T}_1$ and $\widebar{T}_3$ are bounded by estimates \eqref{eq:Cotas_proyeccion_L2:a} and \eqref{eq:Cotas_proyeccion_L2:d}, respectively. We now estimate $\widebar{T}_2$ and $\widebar{T}_4.$ Inequality \eqref{eq:lemma3:cota_de_energia} of Lemma \ref{lemma3:cota_de_energia} implies
 \begin{alignat*}{4}\nonumber
\Thnorm{\errorsigma}^2+\Thnorm{\errorphi}^2+\alpha \pThnorm{\errorsigma-\errorsigmacheck}^2 +\tau \pThnorm{\errorphi-\errorphihat}^2  &\leq C(h^{s-\frac{1}{2}}\Hs{\velocity}+h^{t-\frac{1}{2}}\Ht{\sigma})\pThnorm{\errorsigma-\errorsigmacheck}\\ & \quad + C(h^{s-\frac{1}{2}}\Hs{\velocity} +h^{r-\frac{1}{2}}\Hr{\phi})  \pThnorm{\errorphi-\errorphihat}, 
\end{alignat*} 
and by Cauchy's inequality, for all $\epsilon_1, \epsilon_2 > 0$, we have
 \begin{alignat*}{4}\nonumber
\Thnorm{\errorsigma}^2+\Thnorm{\errorphi}^2+\alpha \pThnorm{\errorsigma-\errorsigmacheck}^2 +\tau \pThnorm{\errorphi-\errorphihat}^2 \\ 
\nonumber \leq \frac{C^2}{4 \epsilon_1}(h^{s-\frac{1}{2}}\Hs{\velocity}+h^{t-\frac{1}{2}}\Ht{\sigma})^2+\epsilon_1 \pThnorm{\errorsigma-\errorsigmacheck}^2\\
+ \frac{C^2}{4 \epsilon_2}(h^{s-\frac{1}{2}}\Hs{\velocity} +h^{r-\frac{1}{2}}\Hr{\phi})^2 + \epsilon_2\pThnorm{\errorphi-\errorphihat}^2. %
\end{alignat*} 
Now, if we take $\epsilon_1 < \alpha$ and $\epsilon_2 < \tau$, then
 \begin{alignat}{4}\nonumber
\Thnorm{\errorsigma}^2+\Thnorm{\errorphi}^2+M_1 \pThnorm{\errorsigma-\errorsigmacheck}^2 +M_2 \pThnorm{\errorphi-\errorphihat}^2 \\ 
\nonumber \leq C(h^{s-\frac{1}{2}}\Hs{\velocity}+h^{t-\frac{1}{2}}\Ht{\sigma}+h^{r-\frac{1}{2}}\Hr{\phi})^2,  %
\end{alignat} 
where $M_i=\alpha-\epsilon_i> 0$, for $i=1,2$, which implies that we can estimate each term of the left-hand side of the last inequality of the following way
\begin{subequations}
 \begin{alignat}{4}\label{eq:Proof_of_Theorem_estimate:subestimate1}
\widebar{T}_2  \leq C(h^{s-\frac{1}{2}}\Hs{\velocity}+h^{t-\frac{1}{2}}\Ht{\sigma} +h^{r-\frac{1}{2}}\Hr{\phi}),  \\ \label{eq:Proof_of_Theorem_estimate:subestimate2}
\widebar{T}_4 \leq C(h^{s-\frac{1}{2}}\Hs{\velocity}+h^{t-\frac{1}{2}}\Ht{\sigma} +h^{r-\frac{1}{2}}\Hr{\phi}),\\ \label{eq:Proof_of_Theorem_estimate:subestimate3}
 \pThnorm{\errorsigma-\errorsigmacheck}\leq C (h^{s-\frac{1}{2}}\Hs{\velocity}+h^{t-\frac{1}{2}}\Ht{\sigma} +h^{r-\frac{1}{2}}\Hr{\phi}),\\ \label{eq:Proof_of_Theorem_estimate:subestimate4}
 \pThnorm{\errorphi-\errorphihat}  \leq C(h^{s-\frac{1}{2}}\Hs{\velocity}+h^{t-\frac{1}{2}}\Ht{\sigma} +h^{r-\frac{1}{2}}\Hr{\phi}). %
\end{alignat}  
By inserting \eqref{eq:Cotas_proyeccion_L2:a},\eqref{eq:Cotas_proyeccion_L2:d} and \eqref{eq:Proof_of_Theorem_estimate:subestimate1}-\eqref{eq:Proof_of_Theorem_estimate:subestimate2} into inequality \eqref{eq:theorem_estimates_fluxes}, we conclude the estimate \eqref{eq:Theorem:A_priori_error_estimate:a} stated in the theorem.
\end{subequations}\\

\subsubsection*{Step 5: Estimates of numerical fluxes of rotational and divergence.} Now, by the triangle inequality, we have
\begin{alignat}{4} \label{eq:triangleinequalitysigmacheck}
    \pThnorm{\sigma- \sigmacheck_h} \leq \pThnorm{\sigma- \PM \sigma} + \pThnorm{\errorsigmacheck},
\end{alignat}
where the first term can be estimated using \eqref{eq:Cotas_proyeccion_L2:c}. For the second term, by applying the triangle inequality, a scaling argument, and the estimates \eqref{eq:Proof_of_Theorem_estimate:subestimate1} and \eqref{eq:Proof_of_Theorem_estimate:subestimate3}, we obtain
\begin{alignat}{4} \nonumber
 \pThnorm{\errorsigmacheck} &\leq \pThnorm{\errorsigma-\errorsigmacheck} + \pThnorm{\errorsigma} \leq \pThnorm{\errorsigma-\errorsigmacheck} + Ch^{-\frac{1}{2}} \Thnorm{\errorsigma} \\ \label{eq:Proof_of_Theorem_estimate:subestimate_psi_check}
 &\leq C\left(h^{s-1} \Hs{\velocity} + h^{t-1} \Ht{\sigma}  + h^{r-1} \Hr{\phi}\right),
\end{alignat}
for $1\leq s,t,r \leq k+1$. From \eqref{eq:Cotas_proyeccion_L2:c} and \eqref{eq:Proof_of_Theorem_estimate:subestimate_psi_check}, we conclude that
\begin{alignat*}{4} 
 \pThnorm{\sigma- \sigmacheck_h}  &\leq C\left(h^{s-1} \Hs{\velocity}+h^{t-1} \Ht{\sigma} +h^{r-1} \Hr{\phi}\right),
\end{alignat*}
and using a similar argument along with estimates \eqref{eq:Cotas_proyeccion_L2:e}, \eqref{eq:Proof_of_Theorem_estimate:subestimate2}, and \eqref{eq:Proof_of_Theorem_estimate:subestimate4}, we obtain
\begin{alignat*}{4} 
 \pThnorm{\phi- \phihat_h}  &\leq C\left(h^{s-1} \Hs{\velocity}+h^{t-1} \Ht{\sigma} +h^{r-1} \Hr{\phi}\right).
\end{alignat*}
Combining the last two inequalities, we arrive at  \eqref{eq:Theorem:A_priori_error_estimate:a2}.
\subsubsection*{Step 6: Estimates of the vector field.}
To prove the last part of the theorem, we assume the regularity condition \eqref{eq:dual-problem:estimate} for the dual problem \eqref{eq:dual-problem}. Applying the triangle inequality, we get  
\begin{alignat}{4}
    \Thnorm{\velocity-\velocity_h} &\leq \Thnorm{\velocity-\PiV \velocity} + \Thnorm{\erroru}. \label{eq:theorem_estimates_triangle_inequality_u}
\end{alignat}  
The first term can be bounded using \eqref{eq:Cotas_proyeccion_L2:g}, so we now turn our attention to estimating the second term. Testing equations \eqref{eq:dual-problem:a}-\eqref{eq:dual-problem:c} with $\errorsigma$, $\errorphi$, $\erroru$ and taking $\Theta = \erroru$, we have  
\begin{alignat}{4}\nonumber
-\Thprod{\errorsigma}{\eta} +\Thprod{\errorsigma}{\rot \,\vect{\Psi}}&=0, \\  \nonumber
-\Thprod{\errorphi}{\Phi} -\Thprod{\errorsigma}{ \divergence\,\vect{\Psi}}&=0, \\ \nonumber \Thprod{\erroru}{
\curl\,\eta}+\Thprod{\erroru}{\grad\,\Phi} &=\Thprod{\erroru}{\erroru}. 
\end{alignat}
Summing the equations above, we obtain
\begin{alignat}{4}\nonumber
    \Thnorm{\erroru}^2&=\Thprod{\erroru}{\curl \ \eta}+\Thprod{\erroru}{\grad   \Phi}-\Thprod{\errorsigma}{\eta}+\Thprod{\errorsigma}{\rot  \vect{\Psi}}\\
    & \quad \  -\Thprod{\errorphi}{\Phi} -\Thprod{\errorsigma}{ \divergence\,\vect{\Psi}}.
\end{alignat}
Adding the projection operators, we have
\begin{alignat}{4}\nonumber
\Thnorm{\erroru}^2    &=\Thprod{\erroru}{\curl (\PiW \eta) }-\Thprod{\erroru}{\curl(\PiW \eta-\eta)}+\Thprod{\erroru}{\grad(\PiW \Phi)}\\
    & \quad -\Thprod{\erroru}{\grad(\PiW \Phi-\Phi)}- \Thprod{\errorsigma}{\eta-\PiW \eta}- \Thprod{\errorsigma}{\PiW \eta} \nonumber\\ \nonumber
        &\quad +\Thprod{\errorsigma}{\rot (\PiV \vect{\Psi}) }-\Thprod{\errorsigma}{\rot(\PiV \vect{\Psi}-\vect{\Psi})}-\Thprod{\errorphi}{ \Phi-\PiW \Phi}\\       
    & \quad -\Thprod{\errorphi}{\PiW \Phi}- \Thprod{\errorphi}{\divergence(\vect{\Psi}-\PiV \vect{\Psi})}- \Thprod{\errorphi}{\divergence (\PiV \vect{\Psi})}.
\end{alignat}
The terms $\Thprod{\errorsigma}{\eta - \PiW \eta} = \Thprod{\errorphi}{ \Phi - \PiW \Phi} = 0$ because the definition of the $L^2$--projection. Now, integrating by parts and using the orthogonality of the $L^2$--projections, 
\begin{alignat}{4}\nonumber
    \Thnorm{\erroru}^2&=\Thprod{\erroru}{\curl (\PiW \eta)}-\pThprod{\erroru \cdot \normal^\perp}{\PiW \eta-\eta}+\Thprod{\erroru}{\grad(\PiW \Phi)}\\
    & \quad -\pThprod{\erroru \cdot \normal}{\PiW \Phi-\Phi}- \Thprod{\errorsigma}{\PiW \eta}+ \Thprod{\errorsigma}{\rot(\PiV  \vect{\Psi})} \nonumber\\ \nonumber
        &\quad +\pThprod{\errorsigma}{ (\PiV  \vect{\Psi}- \vect{\Psi})\cdot \normal^\perp }-\Thprod{\errorphi}{\PiW \Phi}-\pThprod{\errorphi}{ ( \vect{\Psi}-\PiV  \vect{\Psi})\cdot \normal}\\    & \quad -\Thprod{\errorphi}{\divergence (\PiV  \vect{\Psi})}. \label{eq:priori_errror_estimates_u}
\end{alignat}
Now, taking $\chi = \PiW \eta$, $\varphi = \PiW \Phi$, and $\varphi = \PiV \vect{\Psi}$ in equations \eqref{eq:lemma1:equacion_error:a}-\eqref{eq:lemma1:equacion_error:c}, we arrive to
\begin{alignat*}{4} 
\Thprod{\erroru}{\curl \, (\PiW \eta)}-\Thprod{\errorsigma}{\PiW \eta}&= \pThprod{\velocity\cdot \normal^\perp-\velocitycheck_h\cdot \normal^\perp}{\PiW \eta},    \\  
\Thprod{\erroru}{\grad (\PiW \Phi)}-\Thprod{\errorphi}{\PiW \Phi}&=\pThprod{\velocity\cdot \normal-\velocityhat_h\cdot \normal}{\PiW \Phi
}, 
\\  
\Thprod{\errorsigma}{\rot \, (\PiV  \vect{\Psi})}  - 
\Thprod{\errorphi}{\divergence (\PiV  \vect{\Psi})} 
& =-\pThprod{\errorsigmacheck}{\PiV  \vect{\Psi}\cdot \normal^\perp}-\pThprod{\errorphihat}{ \PiV  \vect{\Psi} \cdot \normal },
\end{alignat*}
and substituting above equations in \eqref{eq:priori_errror_estimates_u}, we obtain
\begin{alignat}{4}\nonumber
      \Thnorm{\erroru}^2&=\pThprod{\velocity \cdot \normal^\perp-\velocitycheck_h \cdot \normal^\perp}{\PiW \eta}-\pThprod{\erroru \cdot \normal^\perp}{\PiW \eta-\eta}\\ \nonumber
    &\quad + \pThprod{\velocity \cdot \normal-\velocityhat_h \cdot \normal}{\PiW \Phi}-\pThprod{\erroru \cdot \normal}{\PiW \Phi-\Phi}\\ \nonumber
    &\quad - \pThprod{\errorsigmacheck}{\PiV  \vect{\Psi} \cdot \normal^\perp}-\pThprod{\errorphihat}{\PiV  \vect{\Psi} \cdot \normal}\\ \nonumber
        &\quad + \pThprod{\errorsigma}{(\PiV  \vect{\Psi}-  \vect{\Psi}) \cdot \normal^\perp}-\pThprod{\errorphi}{( \vect{\Psi}-\PiV   \vect{\Psi})\cdot \normal}.
\end{alignat}
Since $\eta$, $\Phi$, $\vect{\Psi} \cdot \normal$, and $\vect{\Psi} \cdot \normal^\perp$ are continuous, they are therefore single-valued.  Thus, from equations \eqref{eq:lemma1:equacion_error:d}-\eqref{eq:lemma1:equacion_error:e}, we obtain
\begin{alignat}{4} \nonumber
  \Thnorm{\erroru}^2 &=\pThprod{\velocity \cdot \normal^\perp-\velocitycheck_h \cdot \normal^\perp-\erroru \cdot \normal^\perp}{\PiW \eta-\eta}+\pThprod{\velocity \cdot \normal-\velocityhat_h \cdot \normal^\perp-\erroru \cdot \normal}{\PiW \Phi-\Phi}\\  \nonumber
  & \quad + \pThprod{\errorsigma-\errorsigmacheck}{(\PiV  \vect{\Psi}- \vect{\Psi})\cdot \normal^\perp}+\pThprod{\errorphi-\errorphihat}{(\PiV  \vect{\Psi}- \vect{\Psi})\cdot \normal}\\
  &:=\widetilde{T}_1+\widetilde{T}_2+\widetilde{T}_3+\widetilde{T}_4.\label{eq:estimates_for_u}
\end{alignat}
In the following, we establish bounds for the terms $\widetilde{T}_1$, $\widetilde{T}_2$, $\widetilde{T}_3$, and $\widetilde{T}_4$. We explicitly provide the bounds for $\widetilde{T}_1$ and $\widetilde{T}_3$, and the bounds for $\widetilde{T}_2$ and $\widetilde{T}_4$ follow by a similar argument.\\
Rewriting the term $\widetilde{T}_1$, we get
\begin{alignat}{4}\nonumber
\widetilde{T}_1&=\pThprod{\velocity \cdot \normal^\perp-\velocitycheck_h \cdot \normal^\perp-\erroru \cdot \normal^\perp}{\PiW \eta-\eta}\\ \nonumber
&=\pThprod{(\velocity-\PiV \velocity) \cdot \normal^\perp}{\PiW \eta-\eta}+\pThprod{ \velocity_h \cdot \normal^\perp - \velocitycheck_h \cdot \normal^\perp }{\PiW \eta-\eta},
\end{alignat}
and using the definitions of numerical fluxes \eqref{eq:MixedFormulation:Weakform:d} and \eqref{eq:MixedFormulation:Weakform:e}, then
\begin{alignat*}{4}\nonumber
\widetilde{T}_1&=\pThprod{(\velocity-\PiV \velocity) \cdot \normal^\perp}{\PiW \eta-\eta}+\pThprod{\alpha(\errorsigma-\errorsigmacheck)}{\PiW \eta-\eta}\\
& \quad +\pThprod{\alpha(\PM \sigma-\PiW \sigma)}{\PiW \eta-\eta},
\end{alignat*}
and by the Cauchy-Schwarz inequality,
\begin{alignat*}{4}\nonumber
\widetilde{T}_1& \leq \pThnorm{\velocity-\PiV \velocity}\pThnorm{\PiW \eta-\eta}+\alpha\pThnorm{\errorsigma-\errorsigmacheck}\pThnorm{\PiW \eta-\eta}\\
& \quad +\alpha\pThnorm{\PM \sigma-\PiW \sigma}\pThnorm{\PiW \eta-\eta}.
\end{alignat*}
Using a standard scaling argument 
\begin{alignat*}{4}\nonumber
\widetilde{T}_1& \leq C\pThnorm{\velocity-\PiV \velocity}h^{-\frac{1}{2}}\Thnorm{\PiW \eta-\eta}+C\pThnorm{\errorsigma-\errorsigmacheck}h^{-\frac{1}{2}}\Thnorm{\PiW \eta-\eta}\\
& \quad +C\pThnorm{\PM \sigma-\PiW \sigma}h^{-\frac{1}{2}}\Thnorm{\PiW \eta-\eta},
\end{alignat*}
and with the estimates for the projection  operators \eqref{eq:Cotas_proyeccion_L2:a}, we obtain
\begin{alignat*}{4}\nonumber
\widetilde{T}_1& \leq C\pThnorm{\velocity-\PiV \velocity}h^{\frac{1}{2}} \| \eta \|_{H^1(\Omega)}+ C\pThnorm{\errorsigma-\errorsigmacheck}h^{\frac{1}{2}}\| \eta \|_{H^1(\Omega)}\\
& \quad + C\pThnorm{\PM \sigma-\PiW \sigma}h^{\frac{1}{2}} \| \eta \|_{H^1(\Omega)}.
\end{alignat*}
Since the bound \eqref{eq:dual-problem:estimate} hold  with $\Theta = \erroru$, we infer
\begin{alignat*}{4}\nonumber
\widetilde{T}_1& \leq  C \pThnorm{\velocity-\PiV \velocity}h^{\frac{1}{2}}\Thnorm{\erroru}+C\pThnorm{\errorsigma-\errorsigmacheck}h^{\frac{1}{2}}\Thnorm{\erroru}\\
& \quad +C\pThnorm{\PM \sigma-\PiW \sigma}h^{\frac{1}{2}}\Thnorm{\erroru}\\
& = C\pThnorm{\velocity-\PiV \velocity}h^{\frac{1}{2}}\Thnorm{\erroru}+C\alpha\pThnorm{\errorsigma-\errorsigmacheck}h^{\frac{1}{2}}\Thnorm{\erroru}\\
& \quad +C(\pThnorm{ \sigma-\PM \sigma }+\pThnorm{ \sigma-\PiW \sigma })h^{\frac{1}{2}}\Thnorm{\erroru}.
\end{alignat*}
Applying the projections bounds \eqref{eq:Cotas_proyeccion_L2:b}, \eqref{eq:Cotas_proyeccion_L2:c} and  \eqref{eq:Cotas_proyeccion_L2:h} along with the estimate  \eqref{eq:Proof_of_Theorem_estimate:subestimate3} gives
\begin{alignat}{4}\nonumber
\widetilde{T}_1& \leq C(h^{s}\Hs{\velocity}+h^{t}\Ht{\sigma} +h^{r}\Hr{\phi})\Thnorm{\erroru}.
\end{alignat}

Furthermore, to estimate the term $\widetilde{T}_3$, we apply the Cauchy–Schwarz inequality, a  scaling argument and the projection bound \eqref{eq:Cotas_proyeccion_L2:g} to obtain
\begin{alignat}{4}\nonumber
\widetilde{T}_3 &= \pThprod{\errorsigma - \errorsigmacheck}{(\PiV  \vect{\Psi} - \vect{\Psi}) \cdot \normal^\perp} \leq  \pThnorm{\errorsigma - \errorsigmacheck} \pThnorm{\PiV  \vect{\Psi} - \vect{\Psi}} \\ \nonumber &\leq  C h^{-\frac{1}{2}}\pThnorm{\errorsigma - \errorsigmacheck} \Thnorm{\PiV  \vect{\Psi} - \vect{\Psi}}\leq h^{\frac{1}{2}} \pThnorm{\errorsigma - \errorsigmacheck} \| \vect{\Psi} \|_{\vect{H}^2(\Omega)},
\end{alignat}
and combining the bounds \eqref{eq:dual-problem:estimate} and \eqref{eq:Proof_of_Theorem_estimate:subestimate3}, we get
\begin{alignat}{4}\nonumber
\widetilde{T}_3 &\leq C \left( h^{s} \Hs{\velocity} + h^{t} \Ht{\sigma} + h^{r} \Hr{\phi} \right) \Thnorm{\erroru}.
\end{alignat}
Finally, we derive the bounds for $\widetilde{T}_1$ and $\widetilde{T}_4$ in a similar manner.  
Thus, by combining these four estimates in \eqref{eq:estimates_for_u}, dividing by $\Thnorm{\erroru} > 0$, and substituting in \eqref{eq:theorem_estimates_triangle_inequality_u}, we obtain the second result in \eqref{eq:Theorem:A_priori_error_estimate:b}, as stated in the theorem.
\end{proof}
\subsection{Hybridization of HDG methods}\label{section:Proofs_Hybridization}
Here, we present the proofs of the theorems stated in the Subsection \ref{subsection:Main Results:Hybridization_of_HDG_methods}, which focus on establishing the existence and uniqueness of the local problems. This enables us to derive a system only in terms of trace unknowns and apply the static condensation technique to reduce the computational cost. Furthermore, we characterize the solution of the HDG method by decomposing it into components that depend on the trace unknowns and $\fb$. Finally, we express the trace unknowns as the minimizer of an energy functional, where the resulting global system corresponds to a symmetric positive-definite matrix.

Here, we present the proof for the method Hybridization of Type III. The other cases are analogous and therefore will be omitted.

\subsubsection*{Proof of Theorem \ref{Teo:Localproblem} (Well-posedness  of the local problem)} 
Since the local problem defines a finite-dimensional square system. That is,  $(\sigma_h, \phi_h, \velocity_h )\in \PK \times \PK \times \PKb,$ such that verify
\begin{subequations} \label{eq:HDG-vector-laplacian-second-formulation:LocalSolver}
\begin{alignat}{4} \label{eq:HDG-vector-laplacian-second-formulation:LocalSolver:a}
\Kprod{\sigma_h}{\chi}-\Kprod{\velocity_h}{\curl \, \chi} & =-\pKprod{ \velocitycheck_h\cdot \normal^\perp}{\chi
},  & \\  \label{eq:HDG-vector-laplacian-second-formulation:LocalSolver:b}
\Kprod{\phi_h}{\varphi}-\Kprod{\velocity_h}{\grad\varphi} & =-\pKprod{\velocityhat_h\cdot \normal}{\varphi
} ,& 
\\ \nonumber
\Kprod{\sigma_h}{\rot \, \zb}+\pKprod{\sigma_h + \alpha^{-1} \velocity_h \cdot \normal^\perp }{\zb\cdot \normal^\perp
}  &- 
\Kprod{\phi_h }{\divergence \zb}+\pThprod{\phi_h+\tau^{-1}\velocity_h \cdot \normal}{\zb \cdot \normal} \\
 = \Kprod{\fb}{\zb}  \label{eq:HDG-vector-laplacian-second-formulation:LocalSolver:c}
+\pKprod{\alpha^{-1}\velocitycheck_h\cdot \normal^\perp}{\zb \cdot \normal^{\perp}} &- 
\pKprod{\alpha^{-1}\velocityhat_h}{ \zb \cdot \normal } , 
\end{alignat}
\end{subequations}
we only have to show that, when we set the
data $\fb=0$ and $\velocitycheck \cdot \normal^\perp=0$ and $\velocityhat\cdot \normal=0$, the only solution is zero. Thus, testing with $\chi=\sigma_h$, $\varphi=\phi_h$ in equations \eqref{eq:HDG-vector-laplacian-second-formulation:LocalSolver:a} and \eqref{eq:HDG-vector-laplacian-second-formulation:LocalSolver:b}, we obtain
\begin{eqnarray*}
\Kprod{\sigma_h}{\sigma_h}-\Kprod{\velocity_h}{\curl \, \sigma_h}  & =0, 
\\ 
\Kprod{\phi_h}{\phi_h}- \Kprod{\velocity_h}{\grad\phi_h} & =0.
\end{eqnarray*}
and integrating by parts, 
\begin{eqnarray*}
\Kprod{\sigma_h}{\sigma_h}-\Kprod{\rot \,\velocity_h}{\sigma_h}-\pKprod{\velocity_h \cdot \normal^{\perp}}{\sigma_h} =0, & \\
\Kprod{\phi_h}{\phi_h}+\Kprod{\divergence \velocity_h}{\phi_h}-\pKprod{\velocity_h\cdot \normal}{\phi_h}  =0,&
\end{eqnarray*}
which, rearranging the terms, is equivalent to
\begin{eqnarray*}
\Kprod{\rot \,\velocity_h}{\sigma_h}  &=& \Kprod{\sigma_h}{\sigma_h} -\pKprod{\velocity_h \cdot \normal^{\perp}}{\sigma_h} ,  \\
\Kprod{\divergence \velocity_h}{\phi_h} &=&- \Kprod{\phi_h}{\phi_h}+ \pKprod{\velocity_h \cdot \normal}{\phi_h}.
\end{eqnarray*}
If we take $\zb = \velocity_h$ in equation \eqref{eq:HDG-vector-laplacian-second-formulation:LocalSolver:c} and use the last two equations, we get
\begin{eqnarray*}
\Kprod{\sigma_h}{\sigma_h} + 
\pKprod{\alpha^{-1} \velocity_h\cdot \normal^\perp}{ \velocity_h\cdot \normal^\perp }+ 
\Kprod{\phi_h}{\phi_h} +
\pKprod{\tau^{-1}\velocity_h \cdot \normal  }{\velocity_h\cdot \normal} = 0.
\end{eqnarray*}
Since $\alpha > 0$ and $\tau > 0$, we obtain $\phi_h = 0$ and $\sigma_h = 0$ in all of $K$. We also have $\velocity_h \cdot \normal^\perp = 0$ and $\velocity_h \cdot \normal = 0$ on $\partial K$ hold, which imply  $\velocity_h= 0$ on $\partial K$. Given that $\rot \, \PKb \subseteq \PK$ and $\divergence \PKb \subseteq \PK$, from \eqref{eq:HDG-vector-laplacian-second-formulation:LocalSolver:a}-\eqref{eq:HDG-vector-laplacian-second-formulation:LocalSolver:b}, we obtain  
\begin{eqnarray*}
\rot \, \velocity_h &=& 0, \quad \text{ in } K,  \\
\divergence \velocity_h &=& 0, \quad \text{ in } K,  \\
\velocity_h &=& 0, \quad \text{ on } \partial K,
\end{eqnarray*}
which implies that $\velocity_h = 0$ in $K$, see \textit{e.g.} \cite[Theorem 4.4]{16_MKrízek}, concluding the proof.

Subsequently, we prove the second main result of Section \ref{subsection:Main Results:Hybridization_of_HDG_methods}.
\subsubsection*{Proof of Theorem \ref{Teo:CharacterizationHDGsolution}  (Characterization of HDG solution)}
The well-posedness of the first and second local problems, given in \eqref{eq:Characterization-of-error:theorem3firstproblem} and \eqref{eq:Characterization-of-error:theorem3secondproblem}, respectively, can be established by an argument similar to the proof of Theorem~\ref{Teo:Localproblem}. Furthermore, by summing each equation of the systems \eqref{eq:Characterization-of-error:theorem3firstproblem} and \eqref{eq:Characterization-of-error:theorem3secondproblem}, we observe that the resulting expressions correspond precisely to those defining the HDG method. Therefore, the solution $(\Sigma, \Phi, \vect{U})$ satisfies the local problem of \eqref{eq:MixedFormulation:Weakform}. It remains to prove the equation stated in \eqref{eq:Characterization-of-error:theorem3energy}. To that end, we rewrite the equations \eqref{eq:Characterization-of-error:theorem3:a}–\eqref{eq:Characterization-of-error:theorem3:b} of the first local problem by replacing the test functions $\etab^t$ and $\etab^n$ with $\mub^t$ and $\mub^n$, respectively. Then, by taking $\chi = \Sigmafourtheta$ and $\varphi = \Phifourtheta$ in those equations, we obtain
\begin{subequations}
\begin{alignat*}{4} 
\Kprod{\Sigmafourthmu}{\Sigmafourtheta}-\Kprod{\Ufourthmu}{\curl \, \Sigmafourtheta}+ \pKprod{ \mub^t\cdot \normal^\perp}{\Sigmafourtheta
} & =0,\\   
\Kprod{\Phifourthmu}{\Phifourtheta}-\Kprod{\Ufourthmu}{\grad \Phifourtheta}+\pKprod{\mub^n \cdot \normal}{\Phifourtheta
} & =0.
\end{alignat*}
Integrating by parts, we have
\begin{alignat}{4}  \label{eq:proof_Theo_charactization:c}
\Kprod{\rot \, \Ufourthmu}{\Sigmafourtheta}&=\Kprod{\Sigmafourthmu}{\Sigmafourtheta}+ \pKprod{ (\mub^t-\Ufourthmu) \cdot \normal^\perp}{\Sigmafourtheta
} ,\\   \label{eq:proof_Theo_charactization:d}
-\Kprod{\divergence \Ufourthmu}{\Phifourtheta}&=\Kprod{\Phifourthmu}{\Phifourtheta}+\pKprod{(\mub^n-\Ufourthmu) \cdot \normal}{\Phifourtheta
} .
\end{alignat}
\end{subequations}
and taking $\zb = \Ufourthmu$ in the third equation \eqref{eq:Characterization-of-error:theorem3:c} of the first local problem, 
\begin{alignat}{4}\nonumber 
\Kprod{\Sigmafourtheta}{\rot \, \Ufourthmu}+\pKprod{\Sigmacheckfourths}{\Ufourthmu \cdot \normal^\perp
}   - 
\Kprod{\Phifourtheta}{\divergence \Ufourthmu}
\\+\pKprod{\Phihatfourths}{\Ufourthmu \cdot \normal}
 = 0. \label{eq:proof_Theo_charactization:e}
\end{alignat}
Replacing  equations \eqref{eq:proof_Theo_charactization:c} and  \eqref{eq:proof_Theo_charactization:d} in \eqref{eq:proof_Theo_charactization:e}, we find that
\begin{alignat}{4}\nonumber 
\Kprod{\Sigmafourthmu}{\Sigmafourtheta}+ \pKprod{ (\mub^t-\Ufourthmu) \cdot \normal^\perp}{\Sigmafourtheta
}+\pKprod{\Sigmacheckfourths }{\Ufourthmu \cdot \normal^\perp
} \\ + \Kprod{\Phifourthmu}{\Phifourtheta}+\pKprod{(\mub^n-\Ufourthmu) \cdot \normal}{\Phifourtheta
}+\pKprod{\Phihatfourths}{\Ufourthmu \cdot \normal}
 = 0, \label{eq:proof_Theo_charactization:f}
\end{alignat}
and considering from \eqref{eq:Characterization-of-error:theorem3:d}-\eqref{eq:Characterization-of-error:theorem3:e}  that $\Sigmacheckfourths=\Sigmafourtheta+\alpha^{-1}(\Ufourtheta-\etab^t)\cdot \normal^\perp $ and $\Phihatfourths=\Phifourtheta+\tau^{-1}(\Ufourtheta -\etab^n)\cdot \normal$, we get
\begin{alignat*}{4}\nonumber
\Kprod{\Sigmafourthmu}{\Sigmafourtheta}+ \Kprod{\Phifourthmu}{\Phifourtheta} 
&+\pKprod{\alpha^{-1} (\Ufourthmu  - \mub^t)\cdot \normal^\perp)}{(\Ufourtheta   - \etab^t)\cdot \normal^\perp}  \\
&+\pKprod{\tau^{-1} (\Ufourthmu   -\mub^n )\cdot \normal)}{(\Ufourtheta   - \etab^n)\cdot \normal} \nonumber \\
&=- \pKprod{\Sigmacheckfourths }{\mub^t\cdot \normal^\perp} - \pKprod{\Phihatfourths}{\mub^n \cdot \normal}, 
\end{alignat*}
and summing over the elements $K\in \Th$, this yields
\begin{alignat}{4}\nonumber
\Thprod{\Sigmafourthmu}{\Sigmafourtheta}+ \Thprod{\Phifourthmu}{\Phifourtheta} 
&+\pThprod{\alpha^{-1} (\Ufourthmu  - \mub^t)\cdot \normal^\perp)}{(\Ufourtheta   - \etab^t)\cdot \normal^\perp}  \\
&+\pThprod{\tau^{-1} (\Ufourthmu   -\mub^n )\cdot \normal)}{(\Ufourtheta   - \etab^n)\cdot \normal} \nonumber \\
&=- \pThprod{\Sigmacheckfourths }{\mub^t\cdot \normal^\perp} - \pThprod{\Phihatfourths}{\mub^n \cdot \normal} \label{eq:proof_Theo_charactization:g}.
\end{alignat}
Now, taking $\chi = \Sigmafourthf$ and $\varphi = \Phifourthf$ in the first two equations that define the first local problem \eqref{eq:Characterization-of-error:theorem3:a}-\eqref{eq:Characterization-of-error:theorem3:b}, and replacing $\etab^t$ and $\etab^n$ with $\mub^t$ and $\mub^n$, we obtain
\begin{subequations}
\begin{alignat*}{4} 
\Kprod{\Sigmafourthmu}{\Sigmafourthf}-\Kprod{\Ufourthmu}{\curl \, \Sigmafourthf}+ \pKprod{\mub^t \cdot \normal^\perp}{\Sigmafourthf} & =0,\\  
\Kprod{\Phifourthmu}{\Phifourthf}-\Kprod{\Ufourthmu}{\grad \Phifourthf}+\pKprod{\mub^n \cdot \normal}{\Phifourthf
} & =0.
\end{alignat*}
\end{subequations}
\begin{subequations}
Integrating by parts, we have
\begin{alignat}{4}  \label{eq:proof_Theo_charactization:j}
\Kprod{\rot \, \Ufourthmu}{\Sigmafourthf}&=\Kprod{\Sigmafourthmu}{\Sigmafourthf}+ \pKprod{ (\mub^t-\Ufourthmu) \cdot \normal^\perp}{\Sigmafourthf
} ,\\   \label{eq:proof_Theo_charactization:k}
-\Kprod{\divergence \Ufourthmu}{\Phifourthf}&=\Kprod{\Phifourthmu}{\Phifourthf}+\pKprod{(\mub^n-\Ufourthmu) \cdot \normal}{\Phifourthf
} ,
\end{alignat}
\end{subequations}
and by setting $\zb = \Ufourthmu$ in the third equation of the second problem \eqref{eq:Characterization-of-error:theorem3:h}, and substituting the terms in \eqref{eq:proof_Theo_charactization:j} and \eqref{eq:proof_Theo_charactization:k}, yield
\begin{alignat*}{4}
\Kprod{\Sigmafourthmu}{\Sigmafourthf}+ \pKprod{ (\mub^t-\Ufourthmu )\cdot \normal^\perp}{\Sigmafourthf
}+ \pKprod{\sigmacheck_\fb }{\Ufourthmu \cdot  \normal^\perp}\\ +\Kprod{\Phifourthmu}{\Phifourthf}+\pKprod{(\mub^n-\Ufourthmu) \cdot \normal}{\Phifourthf} +\pKprod{\widehat{\Phi}_\fb}{\Ufourthmu \cdot \normal}  = \Kprod{\fb}{ \Ufourthmu}.
\end{alignat*}
Considering by \eqref{eq:Characterization-of-error:theorem3:i}-\eqref{eq:Characterization-of-error:theorem3:j} that $\widecheck{\Sigma}_{\fb}=\Sigmasecondf+\alpha^{-1} \Usecondf \cdot \normal^\perp$ and $\widehat{\Phi}_{\fb}=\Phifourthf+\tau^{-1} \Ufourthf \cdot \normal,$ we obtain
\begin{alignat}{4}\nonumber
\Kprod{\Sigmafourthmu}{\Sigmafourthf}+ \pKprod{\mub^t \cdot \normal^\perp}{\Sigmafourthf
}+ \pKprod{\alpha^{-1} \Usecondf \cdot \normal^\perp }{\Ufourthmu \cdot  \normal^\perp}\\ +\Kprod{\Phifourthmu}{\Phifourthf}+\pKprod{\mub^n \cdot \normal}{\Phifourthf} +\pKprod{\tau^{-1} \Ufourthf \cdot \normal}{\Ufourthmu \cdot \normal}  = \Kprod{\fb}{ \Ufourthmu}, 
 \label{eq:proof_Theo_charactization:n}
\end{alignat}
which implies
\begin{alignat}{4}\nonumber
\Kprod{\Sigmafourthmu}{\Sigmafourthf}+ \pKprod{\alpha^{-1} \Usecondf \cdot \normal^\perp }{(\Ufourthmu  -\mub^t) \cdot \normal^\perp} +\Kprod{\Phifourthmu}{\Phifourthf}\\ +\pKprod{\tau^{-1} \Ufourthf \cdot \normal}{(\Ufourthmu-\mub^n)\cdot \normal}  = \Kprod{\fb}{ \Ufourthmu}-\pKprod{\widecheck{\Sigma}_{\fb}}{\mub^t \cdot \normal^\perp}-\pKprod{\widehat{\Phi}_{\fb}}{\mub^n \cdot \normal}. 
 \label{eq:proof_Theo_charactization:o}
\end{alignat}
\begin{subequations}
On the other hand, by taking $\chi = \Sigmafourthmu$ and $\varphi = \Phifourthmu$ in the first two equations that define the second local problem \eqref{eq:Characterization-of-error:theorem3:f}-\eqref{eq:Characterization-of-error:theorem3:g}, we obtain
\begin{alignat*}{4} 
\Kprod{\Sigmafourthf}{\Sigmafourthmu}-\Kprod{\Ufourthf}{\curl \, \Sigmafourthmu}  & =0, \\  
\Kprod{\Phifourthf}{\Phifourthmu}-\Kprod{\Ufourthf}{\grad \Phifourthmu} & =0,
\end{alignat*}
\end{subequations}
\begin{subequations}
and integrating by parts, then
\begin{alignat}{4} \label{eq:Characterization-of-error:theorem3:r}
\Kprod{\rot \, \Ufourthf}{ \Sigmafourthmu}&= \Kprod{\Sigmafourthf}{\Sigmafourthmu}  -\pKprod{\Ufourthf \cdot \normal^\perp }{\Sigmafourthmu}, \\  \label{eq:Characterization-of-error:theorem3:s}
-\Kprod{\divergence \Ufourthf}{ \Phifourthmu}&=\Kprod{\Phifourthf}{\Phifourthmu}-\pKprod{\Ufourthf \cdot \normal }{\Phifourthmu} .
\end{alignat}
\end{subequations}
Taking $\zb = \Ufourthf$ in the third equation \eqref{eq:Characterization-of-error:theorem3:c} of the first local problem and considering $\mub^t$ and $\mub^n$ instead of $\etab^t$ and $\etab^n$, we obtain
\begin{alignat}{4}
\Kprod{\Sigmafourthmu}{\rot \, \Ufourthf}+\pKprod{\Sigmacheckfourthmu }{\Ufourthf\cdot \normal^\perp
}   - 
\Kprod{\Phifourthmu}{\divergence \Ufourthf}
+\pKprod{\Phihatfourthmu}{\Ufourthf \cdot \normal}
& = 0.\label{eq:Characterization-of-error:theorem3:t}
\end{alignat}
If we replace equations \eqref{eq:Characterization-of-error:theorem3:r} and \eqref{eq:Characterization-of-error:theorem3:s} in \eqref{eq:Characterization-of-error:theorem3:t}, then 
\begin{alignat*}{4}
\Kprod{\Sigmafourthf}{\Sigmafourthmu}  -\pKprod{\Ufourthf \cdot \normal^\perp }{\Sigmafourthmu}+\pKprod{\Sigmacheckfourthmu }{\Ufourthf\cdot \normal^\perp
}   - 
\Kprod{\Phifourthf}{\Phifourthmu} \nonumber \\-\pKprod{\Ufourthf \cdot \normal }{\Phifourthmu}  
+\pKprod{\Phihatfourthmu}{\Ufourthf \cdot \normal}
 = 0,
\end{alignat*}
and using the definition of $\Sigmacheckfourthmu$ and  $\Phihatfourthmu$, given in equations \eqref{eq:Characterization-of-error:theorem3:d} and \eqref{eq:Characterization-of-error:theorem3:e}, we have
\begin{alignat}{4}\nonumber
\Kprod{\Sigmafourthmu}{\Sigmafourthf}+ \pKprod{ \Usecondf \cdot \normal^\perp }{\alpha^{-1}(\Ufourthmu-\mub^t) \cdot  \normal^\perp} +\Kprod{\Phifourthmu}{\Phifourthf}\\ +\pKprod{ \Ufourthf \cdot \normal}{\tau^{-1}(\Ufourthmu-\mub^n) \cdot \normal}=0.\label{eq:Characterization-of-error:theorem3:v}
\end{alignat}
Then, replacing the expression \eqref{eq:Characterization-of-error:theorem3:v} into \eqref{eq:proof_Theo_charactization:o}, we find
\begin{alignat*}{4} 
\Kprod{\fb}{ \Ufourthmu}-\pKprod{\widecheck{\Sigma}_{\fb}}{\mub^t \cdot \normal^\perp}-\pKprod{\widehat{\Phi}_{\fb}}{\mub^n \cdot \normal}=0.
\end{alignat*}  
Summing over all the elements $K \in \Th$, we obtain
\begin{alignat}{4} 
\Thprod{\fb}{ \Ufourthmu}=\pThprod{\widecheck{\Sigma}_{\fb}}{\mub^t \cdot \normal^\perp}+\pThprod{\widehat{\Phi}_{\fb}}{\mub^n \cdot \normal}.
 \label{eq:proof_Theo_charactization:w}
\end{alignat}  
Now, using the global equations of the HDG method,
\begin{alignat*}{4} 
\pThprod{\Sigmacheckfourths + \widecheck{\Sigma}_{\fb}}{\mub^t \cdot \normal^\perp }&=0,\\
\pThprod{\Phihatfourths + \widehat{\Phi}_{\fb}}{\mub^n \cdot \normal }&=0,
\end{alignat*}  
and substituting into \eqref{eq:proof_Theo_charactization:w}, we follow
\begin{alignat}{4} 
\Thprod{\fb}{ \Ufourthmu}=-\pThprod{\Sigmacheckfourths}{\mub^t \cdot \normal^\perp}-\pThprod{\Phihatfourths}{\mub^n \cdot \normal}.
 \label{eq:proof_Theo_charactization:w1}
\end{alignat}  
Combining this last identity with \eqref{eq:proof_Theo_charactization:g},  we obtain the result stated in  \eqref{eq:Characterization-of-error:theorem3energy}.

To conclude the proof, observe that the bilinear form $a_h(\cdot,\cdot)$ is symmetric and positive definite. Consequently, the last result implies that the pair $(\velocitycheck_h,\velocityhat_h)$ corresponds to the minimizer of the total energy functional given in \eqref{eq:Energy}.
\section{Numerical Experiments}\label{section:NumericalExperiments}
In this section, we present numerical experiments designed to validate the theoretical convergence properties of the proposed HDG methods. Each experiment corresponds to a different type of boundary condition: electric, magnetic, and Dirichlet. We use triangles and the squares. We construct the meshes by partitioning the domain $\Omega=(0,1)^2$ into squares of side length $2^{-l}$, with $l \ge 2$. To use triangles, we further subdivide the squares into two triangular elements.

For each experiment, we compute the errors 
\[
e_\sigma:=\Thnorm{\sigma-\sigma_h}, e_\phi:=\Thnorm{\phi-\phi_h}, e_\velocity:=\Thnorm{\velocity-\velocity_h},
e_{\sigmacheck}:=\pThnorm{\sigma-\sigmacheck_h}, e_{\phihat}:=\pThnorm{\phi-\phihat_h}.
\]
 For each error, we also calculate the numerical order of convergence  $\text{e.o.c.}:=\frac{\log \left(e_{\mathcal{T}_1} / e_{\mathcal{T}_2}\right)}{\log \left(h_1 / h_2\right)},$ where $e_{\mathcal{T}_1}$ and $e_{\mathcal{T}_2}$ are the errors computed on two consecutive meshes with step sizes $h_1$ and $h_2$, respectively. 

\begin{table}[h] 
\centering
\tiny
\begin{tabular}{@{}lcl@{\hskip .1in}cc@{\hskip .1in}cc@{\hskip .1in}cc@{\hskip .1in}c@{\hskip .1in}cc@{\hskip .1in}cc@{\hskip .1in}cc@{\hskip .1in}cc}
\toprule
&&  & \multicolumn{2}{c}{$\sigma_h$} && \multicolumn{2}{c}{$\velocity_h$} && \multicolumn{2}{c}{${\phi}_{h}$} && \multicolumn{2}{c}{${\sigmacheck}_{h}$} && \multicolumn{2}{c}{$\phihat_h$}  \\
\cmidrule(lr){4-5} \cmidrule(lr){7-8} \cmidrule(lr){10-11} \cmidrule(lr){13-14} \cmidrule(lr){16-17} 
$k$
& $h$     &&  $e_\sigma$   &e.o.c.&&  $e_\velocity$   &e.o.c.&&  $e_{\phi}$   & e.o.c. && $e_{\sigmacheck}$   & e.o.c. && $e_{\phihat}$ & e.o.c. \\
\midrule
\multicolumn{17}{c}{\normalsize{Triangles\phantom{$R^\big|q_|$}}}
 \\
\midrule
\multirow{6}{*}{0}
&5.00e-01 && 9.70e-01 & -    && 2.67e+00 & -    && 2.73e+00 & -    && 4.74e+00 & -    && 9.11e+00 & - \\
&2.50e-01 && 4.68e-01 & 1.05 && 1.64e+00 & 0.71 && 1.40e+00 & 0.97 && 2.98e+00 & 0.67 && 6.92e+00 & 0.40\\
&1.25e-01 && 2.28e-01 & 1.03 && 8.94e-01 & 0.87 && 6.84e-01 & 1.03 && 1.93e+00 & 0.63 && 4.91e+00 & 0.50\\
&6.25e-02 && 1.12e-01 & 1.03 && 4.64e-01 & 0.95 && 3.35e-01 & 1.03 && 1.30e+00 & 0.58 && 3.45e+00 & 0.51 \\
&3.12e-02 && 5.55e-02 & 1.01 && 2.36e-01 & 0.98 && 1.66e-01 & 1.02 && 8.91e-01 & 0.54 && 2.42e+00 & 0.51 \\
&1.56e-02 && 2.76e-02 & 1.01 && 1.19e-01 & 0.99 && 8.24e-02 & 1.01 && 6.22e-01 & 0.52 && 1.71e+00 & 0.50\\ \hline
\multirow{6}{*}{1} 
&5.00e-01 && 3.14e-01 & -    && 6.89e-01 & -    && 7.58e-01 & -    && 2.02e+00 & -    && 2.91e+00 & - \\
&2.50e-01 && 8.02e-02 & 1.97 && 2.01e-01 & 1.77 && 1.97e-01 & 1.94 && 6.41e-01 & 1.66 && 1.13e+00 & 1.36 \\
&1.25e-01 && 1.93e-02 & 2.05 && 5.53e-02 & 1.86 && 4.91e-02 & 2.01 && 1.92e-01 & 1.73 && 4.05e-01 & 1.48 \\
&6.25e-02 && 4.69e-03 & 2.04 && 1.43e-02 & 1.95 && 1.22e-02 & 2.01 && 6.07e-02 & 1.67 && 1.43e-01 & 1.51 \\
&3.12e-02 && 1.15e-03 & 2.02 && 3.64e-03 & 1.98 && 3.03e-03 & 2.01 && 2.01e-02 & 1.59 && 5.02e-02 & 1.51 \\
&1.56e-02 && 2.86e-04 & 2.01 && 9.16e-04 & 1.99 && 7.56e-04 & 2.00  && 6.87e-03 & 1.55 && 1.77e-02 & 1.50 \\\hline
\multirow{6}{*}{2} 
&5.00e-01 && 6.85e-02 & -    && 1.78e-01 & -    && 1.65e-01 & -    && 2.41e-01 & -   && 2.68e-01 & - \\
&2.50e-01 && 8.82e-03 & 2.96 && 2.25e-02 & 2.98 && 2.11e-02 & 2.97 && 5.61e-02 & 2.1 && 5.60e-02 & 2.26 \\
&1.25e-01 && 1.12e-03 & 2.97 && 2.66e-03 & 3.08 && 2.66e-03 & 2.98 && 1.17e-02 & 2.26 && 1.31e-02 & 2.09 \\
&6.25e-02 && 1.42e-04 & 2.99 && 3.21e-04 & 3.05 && 3.34e-04 & 2.99 && 2.22e-03 & 2.40&& 2.61e-03 & 2.33 \\
&3.12e-02 && 1.78e-05 & 2.99 && 3.95e-05 & 3.02 && 4.19e-05 & 3.00 && 4.04e-04 & 2.46 && 4.84e-04 & 2.43 \\
&1.56e-02 && 2.22e-06 & 3.00 && 4.90e-06 & 3.01 && 5.25e-06 & 3.00 && 7.23e-05 & 2.48 && 8.74e-05 & 2.47  \\\hline
\multirow{6}{*}{3}  
&5.00e-01 && 1.34e-02 & -    && 2.84e-02 & -    && 2.90e-02 & -    && 6.44e-02 & -    && 1.11e-01 & - \\
&2.50e-01 && 8.27e-04 & 4.02 && 2.00e-03 & 3.83 && 1.92e-03 & 3.92 && 5.84e-03 & 3.46 && 1.12e-02 & 3.31 \\
&1.25e-01 && 4.93e-05 & 4.07 && 1.34e-04 & 3.90  && 1.21e-04 & 3.99 && 4.94e-04 & 3.56 && 1.06e-03 & 3.40\\
&6.25e-02 && 2.98e-06 & 4.05 && 8.66e-06 & 3.95 && 7.51e-06 & 4.01 && 4.25e-05 & 3.54 && 9.64e-05 & 3.46 \\
&3.12e-02 && 1.83e-07 & 4.03 && 5.49e-07 & 3.98 && 4.67e-07 & 4.01 && 3.70e-06 & 3.52 && 8.62e-06 & 3.48 \\
&1.56e-02 && 1.13e-08 & 4.01 && 3.45e-08 & 3.99 && 2.91e-08 & 4.00 && 3.24e-07 & 3.51 && 7.65e-07 & 3.49 \\
\hline
\multicolumn{17}{c}{\normalsize{Squares\phantom{$R^\big|q_|$}}}
 \\
\midrule
\multirow{6}{*}{0}
&5.00e-01 && 1.08e+00 & -    && 3.10e+00 & -    && 3.24e+00 & -    && 5.69e+00 & -    && 6.55e+00 & - \\
&2.50e-01 && 5.87e-01 & 0.88 && 2.06e+00 & 0.59 && 1.76e+00 & 0.88 && 3.55e+00 & 0.68 && 4.89e+00 & 0.42 \\
&1.25e-01 && 2.88e-01 & 1.03 && 1.19e+00 & 0.79 && 8.63e-01 & 1.03 && 2.02e+00 & 0.82 && 3.35e+00 & 0.55 \\
&6.25e-02 && 1.40e-01 & 1.04 && 6.32e-01 & 0.92 && 4.20e-01 & 1.04 && 1.19e+00 & 0.76 && 2.30e+00 & 0.54 \\
&3.12e-02 && 6.88e-02 & 1.02 && 3.24e-01 & 0.97 && 2.06e-01 & 1.02 && 7.53e-01 & 0.66 && 1.60e+00 & 0.52 \\
&1.56e-02 && 3.41e-02 & 1.01 && 1.64e-01 & 0.98 && 1.02e-01 & 1.01 && 5.02e-01 & 0.59 && 1.12e+00 & 0.51 \\ \hline
\multirow{6}{*}{1} 
&5.00e-01 && 4.46e-01 & -    && 9.28e-01 & -    && 1.10e+00 & -    && 1.93e+00 & -    && 2.03e+00 & - \\
&2.50e-01 && 1.11e-01 & 2.01 && 2.18e-01 & 2.09 && 2.97e-01 & 1.89 && 5.98e-01 & 1.69 && 6.16e-01 & 1.72 \\
&1.25e-01 && 2.66e-02 & 2.05 && 5.37e-02 & 2.02 && 7.37e-02 & 2.01 && 1.78e-01 & 1.75 && 1.72e-01 & 1.84 \\
&6.25e-02 && 6.49e-03 & 2.04 && 1.35e-02 & 1.99 && 1.82e-02 & 2.02 && 5.59e-02 & 1.67 && 5.69e-02 & 1.6 \\
&3.12e-02 && 1.60e-03 & 2.02 && 3.39e-03 & 1.99 && 4.53e-03 & 2.01 && 1.85e-02 & 1.6  && 2.04e-02 & 1.48 \\
&1.56e-02 && 3.98e-04 & 2.01 && 8.52e-04 & 1.99 && 1.13e-03 & 2.0  && 6.30e-03 & 1.55 && 7.35e-03 & 1.47 \\            \hline
\multirow{6}{*}{2} 
&5.00e-01 && 1.04e-01 & -    && 2.87e-01 & -    && 2.76e-01 & -    && 4.90e-01 & -    && 3.40e-01 & - \\
&2.50e-01 && 1.32e-02 & 2.98 && 3.03e-02 & 3.24 && 3.58e-02 & 2.95 && 1.19e-01 & 2.04 && 9.79e-02 & 1.80 \\
&1.25e-01 && 1.59e-03 & 3.05 && 3.59e-03 & 3.08 && 4.54e-03 & 2.98 && 2.13e-02 & 2.48 && 1.93e-02 & 2.35 \\
&6.25e-02 && 1.96e-04 & 3.02 && 4.43e-04 & 3.02 && 5.70e-04 & 2.99 && 3.76e-03 & 2.50 && 3.55e-03 & 2.44 \\
&3.12e-02 && 2.43e-05 & 3.01 && 5.51e-05 & 3.01 && 7.14e-05 & 3.00 && 6.64e-04 & 2.50 && 6.41e-04 & 2.47 \\
&1.56e-02 && 3.02e-06 & 3.01 && 6.88e-06 & 3.00 && 8.93e-06 & 3.00 && 1.18e-04 & 2.50 && 1.14e-04 & 2.49 \\      \hline
\multirow{6}{*}{3}  
&5.00e-01 && 2.46e-02 & -    && 4.52e-02 & -    && 5.32e-02 & -    && 1.30e-01 & -    && 1.53e-01 & - \\
&2.50e-01 && 1.87e-03 & 3.72 && 3.58e-03 & 3.66 && 3.79e-03 & 3.81 && 1.52e-02 & 3.10 && 1.49e-02 & 3.36 \\
&1.25e-01 && 1.10e-04 & 4.08 && 2.80e-04 & 3.67 && 2.42e-04 & 3.97 && 1.26e-03 & 3.60 && 1.11e-03 & 3.75 \\
&6.25e-02 && 6.09e-06 & 4.18 && 1.93e-05 & 3.86 && 1.49e-05 & 4.02 && 9.50e-05 & 3.73 && 7.50e-05 & 3.88 \\
&3.12e-02 && 3.49e-07 & 4.12 && 1.25e-06 & 3.95 && 9.17e-07 & 4.02 && 7.50e-06 & 3.66 && 5.33e-06 & 3.81 \\
&1.56e-02 && 2.09e-08 & 4.06 && 7.90e-08 & 3.98 && 5.69e-08 & 4.01 && 6.25e-07 & 3.59 && 4.11e-07 & 3.70 \\
\bottomrule
\end{tabular}
\caption{Experiment \ref{sec: Experiment1}. Electric boundary conditions.}
\label{tab:HistoryConvergenceExperiment1}
\end{table}
\subsection{Experiment 1 (Electric boundary condition)} \label{sec: Experiment1}
In this first experiment, we consider the same test case studied in the context of conforming mixed methods of \cite{23_Arnold,18_HWang}.  We  choose the source term $\fb$ such that the exact solution is
\begin{alignat*}{4}
 \velocity(x, y)&=\left[\begin{array}{c}
\cos(\pi x)\sin(\pi y) \\
  2\sin(\pi x) \cos(\pi y) 
\end{array}\right],\quad
\phi(x,y)&=3\pi\sin(\pi x)\sin(\pi y), \quad
\sigma(x,y)&=\pi\cos(\pi x)\cos(\pi y).
\end{alignat*} 
In Table \ref{tab:HistoryConvergenceExperiment1}, we observe optimal convergence rates for the approximate rotational $\sigma_h$, divergence $\phi_h$, the vector field $\velocity_h$, and the numerical fluxes $\sigmacheck_h$ and $\phihat_h$ for polynomial degrees $k = 0, 1, 2, 3$, over a sequence of uniformly refined triangular and square meshes with mesh size $h$. The errors  $e_\sigma$, $e_\phi$, and $e_\velocity$ converge with optimal order of convergence $k+1$, while the errors $e_{\sigmacheck}$ and $e_{\phihat}$ show an order of $k + 1/2$.  

In \cite[Section 2]{23_Arnold}, the same experiment is carried out for the case $k = 2$, using Raviart--Thomas finite elements of degree $2$ (for triangles) for the vector field $\velocity$ and continuous Lagrange elements of degree $2$ for the auxiliary variable $\sigma = \rot \, \velocity$. The numerical results exhibit optimal convergence of order $2$ for the vector field and order $3$ for the auxiliary variable. Similarly, in \cite[Section 5.1]{18_HWang}, the case $k = 1$ is considered, employing BDM elements of degree $1$ for the vector field $\velocity$, and continuous Lagrange elements of degree $1$ for the auxiliary variables $\sigma$ and $\phi$. The results show optimal convergence of order $2$ for both the vector field and the auxiliary variables. In comparison, our HDG method with $k = 2$ achieves optimal convergence $3$ for $\velocity_h$, $\sigma_h$, and $\phi_h$, which is competitive with the approaches discussed above.


\subsection{Experiment 2 (Magnetic boundary condition)} \label{sec: Experiment2}
 We choose $\fb$ so that the exact solution is 
\begin{alignat*}{4}
 \velocity(x, y)&=\left[\begin{array}{c}
\sin(2\pi x)\cos(\pi y) \\
  2 \cos(2\pi x) \sin(\pi y) 
\end{array}\right],\quad 
\phi(x,y)=-4\pi\cos(2\pi x)\cos(\pi y), \\ 
\sigma(x,y)&=-3\pi\sin(2\pi x)\sin(\pi y).
\end{alignat*} 
As shown in Table~\ref{tab:HistoryConvergenceExperiment2}, the HDG method exhibits optimal convergence of order \( k+1 \) for the volumetric variables \( \sigma_h \), \( \phi_h \), and \( \velocity_h \), and order \( k+1/2 \) for the numerical fluxes \( \sigmacheck_h \) and \( \phihat_h \).
\begin{table}[h] 
\centering
\tiny
\begin{tabular}{@{}lcl@{\hskip .1in}cc@{\hskip .1in}cc@{\hskip .1in}cc@{\hskip .1in}c@{\hskip .1in}cc@{\hskip .1in}cc@{\hskip .1in}cc@{\hskip .1in}cc}
\toprule
&&  & \multicolumn{2}{c}{$\sigma_h$} && \multicolumn{2}{c}{$\velocity_h$} && \multicolumn{2}{c}{${\phi}_{h}$} && \multicolumn{2}{c}{${\sigmacheck}_{h}$} && \multicolumn{2}{c}{$\phihat_h$}  \\
\cmidrule(lr){4-5} \cmidrule(lr){7-8} \cmidrule(lr){10-11} \cmidrule(lr){13-14} \cmidrule(lr){16-17} 
$k$
& $h$     &&  $e_\sigma$   &e.o.c.&&  $e_\velocity$   &e.o.c.&&  $e_{\phi}$   & e.o.c. && $e_{\sigmacheck}$   & e.o.c. && $e_{\phihat}$ & e.o.c. \\
\hline
\multicolumn{17}{c}{\normalsize{Triangles\phantom{$R^\big|q_|$}}}
 \\
\midrule
\multirow{6}{*}{0}
&5.00e-01 && 5.07e+00 & -    && 5.54e+00 & -    && 5.83e+00 & -    && 1.73e+01 & -    && 2.17e+01 & - \\
&2.50e-01 && 2.21e+00 & 1.20  && 4.15e+00 & 0.42 && 2.75e+00 & 1.08 && 1.18e+01 & 0.56 && 1.39e+01 & 0.64 \\
&1.25e-01 && 1.05e+00 & 1.06 && 2.33e+00 & 0.83 && 1.39e+00 & 0.99 && 7.98e+00 & 0.56 && 1.01e+01 & 0.46 \\
&6.25e-02 && 5.14e-01 & 1.04 && 1.22e+00 & 0.93 && 6.83e-01 & 1.02 && 5.49e+00 & 0.54 && 7.15e+00 & 0.50 \\
&3.12e-02 && 2.54e-01 & 1.02 && 6.23e-01 & 0.97 && 3.38e-01 & 1.01 && 3.83e+00 & 0.52 && 5.03e+00 & 0.51 \\
&1.56e-02 && 1.26e-01 & 1.01 && 3.14e-01 & 0.99 && 1.68e-01 & 1.01 && 2.69e+00 & 0.51 && 3.55e+00 & 0.50 \\
 \hline
\multirow{6}{*}{1} 
&5.00e-01 && 2.39e+00 & -    &&  2.56e+00 & -   && 1.79e+00 & -    && 8.16e+00 & -    && 7.31e+00 & - \\
&2.50e-01 && 4.78e-01 & 2.32 && 7.59e-01 & 1.75 && 6.13e-01 & 1.55 && 3.05e+00 & 1.42 && 3.64e+00 & 1.01 \\
&1.25e-01 && 1.20e-01 & 1.99 && 2.10e-01 & 1.86 && 1.56e-01 & 1.98 && 1.04e+00 & 1.55 && 1.30e+00 & 1.49 \\
&6.25e-02 && 2.97e-02 & 2.02 && 5.48e-02 & 1.94 && 3.88e-02 & 2.00  && 3.53e-01 & 1.56 && 4.53e-01 & 1.52 \\
&3.12e-02 && 7.38e-03 & 2.01 && 1.39e-02 & 1.98 && 9.66e-03 & 2.01 && 1.22e-01 & 1.53 && 1.59e-01 & 1.51 \\
&1.56e-02 && 1.84e-03 & 2.01 && 3.50e-03 & 1.99 && 2.41e-03 & 2.00  && 4.26e-02 & 1.52 && 5.59e-02 & 1.51 \\
\hline
\multirow{6}{*}{2} 
&5.00e-01 && 5.97e-01 & -    && 9.69e-01 & -    && 8.05e-01 & -    && 2.06e+00 & -    && 2.09e+00 & - \\
&2.50e-01 && 7.69e-02 & 2.96 && 1.38e-01 & 2.81 && 9.98e-02 & 3.01 && 2.28e-01 & 3.18 && 2.22e-01 & 3.23 \\
&1.25e-01 && 9.82e-03 & 2.97 && 1.62e-02 & 3.09 && 1.28e-02 & 2.96 && 5.79e-02 & 1.98 && 6.08e-02 & 1.87 \\
&6.25e-02 && 1.24e-03 & 2.98 && 1.92e-03 & 3.08 && 1.63e-03 & 2.98 && 1.19e-02 & 2.28 && 1.30e-02 & 2.23 \\
&3.12e-02 && 1.56e-04 & 2.99 && 2.33e-04 & 3.04 && 2.04e-04 & 2.99 && 2.22e-03 & 2.42 && 2.45e-03 & 2.40 \\
&1.56e-02 && 1.95e-05 & 3.00  && 2.88e-05 & 3.02 && 2.56e-05 & 3.00  && 4.00e-04 & 2.47 && 4.45e-04 & 2.46 \\
\hline
\multirow{6}{*}{3}  
&5.00e-01 && 1.47e-01 & -    && 2.27e-01 & -    && 1.77e-01 & -    && 5.25e-01 & -    && 7.04e-01 & - \\
&2.50e-01 && 1.04e-02 & 3.82 && 1.68e-02 & 3.76 && 1.34e-02 & 3.72 && 6.31e-02 & 3.06 && 7.52e-02 & 3.23 \\
&1.25e-01 && 6.56e-04 & 3.99 && 1.13e-03 & 3.89 && 8.66e-04 & 3.95 && 6.00e-03 & 3.39 && 7.19e-03 & 3.39 \\
&6.25e-02 && 4.06e-05 & 4.01 && 7.33e-05 & 3.95 && 5.42e-05 & 4.00  && 5.41e-04 & 3.47 && 6.59e-04 & 3.45 \\
&3.12e-02 && 2.52e-06 & 4.01 && 4.64e-06 & 3.98 && 3.38e-06 & 4.00  && 4.80e-05 & 3.50  && 5.90e-05 & 3.48 \\
&1.56e-02 && 1.57e-07 & 4.01 && 2.91e-07 & 3.99 && 2.11e-07 & 4.00  && 4.24e-06 & 3.50  && 5.24e-06 & 3.49 \\
\hline
\multicolumn{17}{c}{\normalsize{Squares\phantom{$R^\big|q_|$}}}
 \\
\midrule
\multirow{6}{*}{0}
&5.00e-01 && 6.01e+00 & -    && 5.01e+00 & -     && 6.28e+00 & -    && 1.87e+01 & -    && 1.56e+01 & - \\
&2.50e-01 && 2.92e+00 & 1.04 && 5.14e+00 & -0.03 && 3.26e+00 & 0.94 && 1.14e+01 & 0.71 && 9.38e+00 & 0.73 \\
&1.25e-01 && 1.38e+00 & 1.08 && 3.04e+00 & 0.76 && 1.68e+00 & 0.96 && 6.79e+00 & 0.74 && 6.60e+00 & 0.51 \\
&6.25e-02 && 6.53e-01 & 1.08 && 1.62e+00 & 0.90  && 8.35e-01 & 1.01 && 4.14e+00 & 0.71 && 4.62e+00 & 0.51 \\
&3.12e-02 && 3.16e-01 & 1.04 && 8.33e-01 & 0.96 && 4.14e-01 & 1.01 && 2.71e+00 & 0.61 && 3.25e+00 & 0.51 \\
&1.56e-02 && 1.56e-01 & 1.02 && 4.20e-01 & 0.99 && 2.06e-01 & 1.01 && 1.84e+00 & 0.55 && 2.29e+00 & 0.50 \\
 \hline
\multirow{6}{*}{1} 
&5.00e-01 && 2.16e+00 & -    && 3.92e+00 & -    && 2.87e+00 & -    && 6.72e+00 & -    && 5.82e+00 & - \\
&2.50e-01 && 7.00e-01 & 1.63 && 9.34e-01 & 2.07 && 8.72e-01 & 1.72 && 2.71e+00 & 1.31 && 2.31e+00 & 1.33 \\
&1.25e-01 && 1.74e-01 & 2.00  && 2.40e-01 & 1.96 && 2.23e-01 & 1.96 && 7.30e-01 & 1.89 && 7.06e-01 & 1.71 \\
&6.25e-02 && 4.27e-02 & 2.03 && 6.03e-02 & 1.99 && 5.55e-02 & 2.01 && 2.14e-01 & 1.77 && 2.22e-01 & 1.67 \\
&3.12e-02 && 1.06e-02 & 2.02 && 1.51e-02 & 1.99 && 1.38e-02 & 2.01 && 6.97e-02 & 1.62 && 7.57e-02 & 1.55 \\
&1.56e-02 && 2.63e-03 & 2.01 && 3.79e-03 & 2.00  && 3.45e-03 & 2.00  && 2.39e-02 & 1.54 && 2.67e-02 & 1.50 \\
\hline
\multirow{6}{*}{2} 
&5.00e-01 && 9.92e-01 & -    && 1.03e+00 & -    && 1.05e+00 & -    && 3.26e+00 & -    && 1.63e+00 & - \\
&2.50e-01 && 1.26e-01 & 2.98 && 1.69e-01 & 2.61 && 1.61e-01 & 2.71 && 5.49e-01 & 2.57 && 5.83e-01 & 1.49 \\
&1.25e-01 && 1.58e-02 & 3.00 && 2.09e-02 & 3.01 && 2.06e-02 & 2.97 && 1.05e-01 & 2.38 && 1.04e-01 & 2.49 \\
&6.25e-02 && 1.97e-03 & 3.00 && 2.61e-03 & 3.00 && 2.59e-03 & 2.99 && 1.86e-02 & 2.50 && 1.81e-02 & 2.52 \\
&3.12e-02 && 2.45e-04 & 3.00 && 3.26e-04 & 3.00 && 3.24e-04 & 3.00 && 3.27e-03 & 2.51 && 3.21e-03 & 2.50 \\
&1.56e-02 && 3.06e-05 & 3.00 && 4.08e-05 & 3.00 && 4.06e-05 & 3.00 && 5.76e-04 & 2.50 && 5.69e-04 & 2.50 \\
\hline
\multirow{6}{*}{3}  
&5.00e-01 && 2.22e-01 & -    && 3.36e-01 & -    && 3.39e-01 & -   &&  8.60e-01 & -   && 7.01e-01 & - \\
&2.50e-01 && 1.90e-02 & 3.54 && 2.34e-02 & 3.85 && 2.43e-02 & 3.8 && 9.23e-02 & 3.22 && 9.74e-02 & 2.85 \\
&1.25e-01 && 1.26e-03 & 3.91 && 1.78e-03 & 3.72 && 1.61e-03 & 3.92 && 8.09e-03 & 3.51 && 8.79e-03 & 3.47 \\
&6.25e-02 && 7.65e-05 & 4.05 && 1.28e-04 & 3.80 && 9.97e-05 & 4.01 && 5.82e-04 & 3.8  && 6.52e-04 & 3.75 \\
&3.12e-02 && 4.63e-06 & 4.05 && 8.41e-06 & 3.92 && 6.14e-06 & 4.02 && 4.24e-05 & 3.78 && 4.93e-05 & 3.73 \\
&1.56e-02 && 2.85e-07 & 4.02 && 5.35e-07 & 3.97 && 3.81e-07 & 4.01 && 3.39e-06 & 3.65 && 4.02e-06 & 3.61 \\
\bottomrule
\end{tabular}
\caption{Experiment \ref{sec: Experiment2}. Magnetic boundary conditions.}
\label{tab:HistoryConvergenceExperiment2}
\end{table}
\subsection{Experiment 3 (Dirichlet boundary condition)}\label{sec: Experiment3} In this final set of experiments, we investigate the experiment presented in \cite[Section 2]{23_Arnold} and \cite[Section 4]{18_HWang}.  In this simulation, $\fb$ is chosen so that the exact solution is
\begin{alignat*}{4}
 \velocity(x, y)&=\left[\begin{array}{c}
\sin(\pi x)\sin(\pi y) \\ \sin(\pi x)\sin(\pi y)  
\end{array}\right],\quad \phi(x,y)=-\pi\cos(\pi x)\sin(\pi y)-\pi\sin(\pi x)\cos(\pi y), \\\sigma(x,y)&=\pi\cos(\pi x)\sin(\pi y)-\pi \sin(\pi x)\cos(\pi y).
\end{alignat*} 

In Table \ref{tab:HistoryConvergenceExperiment3}, we report optimal convergence rates for the approximations of the vector field $\velocity_h$ and the auxiliary variables $\sigma_h$ and $\phi_h$, all achieving optimal order $k+1$. In addition, we observe that the numerical fluxes $\sigmacheck_h$ and $\phihat_h$ converge at order $k + 1/2$ for $k = 0, 1, 2, 3$. For the case $k = 0$, the asymptotic rate is not captured on coarse meshes. Therefore, in this case only, the experiments are initialized with a finer mesh size. The approximation of the vector field confirms the convergence rate predicted by \eqref{eq:Theorem:A_priori_error_estimate:b} of Theorem \ref{Teo:A_priori_error_estimates}. The auxiliary variables $\sigma_h$ and $\phi_h$ exhibit convergence rates that exceed the theoretical prediction in \eqref{eq:Theorem:A_priori_error_estimate:a}, showing an additional order in $1/2$. We observe a similar improvement for the numerical fluxes $\sigmacheck_h$ and $\phihat_h$, which exceed the bound provided in \eqref{eq:Theorem:A_priori_error_estimate:a2} by the same margin. These results confirm that the numerical experiment achieves optimal convergence rates for all variables, outperforming the theoretical predictions presented in Section~\ref{subsec:MainResults:Apriori_error_estimates}, and suggesting the possibility of improving the \textit{a priori} error estimates.

The same experiment was also carried out in \cite{23_Arnold} and \cite{18_HWang} on triangles. In \cite{23_Arnold}, a mixed formulation involving the rotational variable $\sigma = \rot \, \velocity$ is considered. The authors show that conforming finite element spaces do not preserve the inclusion of solution spaces under differential operators in the De Rham complex when Dirichlet boundary conditions are imposed, resulting in suboptimal convergence for the auxiliary variable, although optimal convergence is maintained for the vector field. They used Raviart--Thomas elements of degree $2$ for $\velocity_h$, and continuous Lagrange elements of degree $2$ for $\sigma_h$. The observed convergence rate for $\velocity_h$ is optimal of order $2$, while the order for $\sigma_h$ is $3/2$. In \cite{18_HWang}, both the rotational and divergence variables are introduced as auxiliary variables. The authors used Brezzi--Douglas--Marini elements of degree $1$ for $\velocity_h$, and continuous Lagrange elements of degree $2$ for the auxiliary variables $\sigma_h$ and $\phi_h$. The vector field converges with optimal order $2$, while the auxiliary variables exhibit convergence rates close to $3/2$, consistent with the findings in \cite{23_Arnold}. Although these formulations are different, our HDG method shows better convergence rates for auxiliary variables in the case $k = 2$, achieving order $3$, which is better than those obtained with the conforming finite element approaches discussed above.

\begin{table}[h] 
\centering
\tiny
\begin{tabular}{@{}lcl@{\hskip .1in}cc@{\hskip .1in}cc@{\hskip .1in}cc@{\hskip .1in}c@{\hskip .1in}cc@{\hskip .1in}cc@{\hskip .1in}cc@{\hskip .1in}cc}
\toprule
&&  & \multicolumn{2}{c}{$\sigma_h$} && \multicolumn{2}{c}{$\velocity_h$} && \multicolumn{2}{c}{${\phi}_{h}$} && \multicolumn{2}{c}{${\sigmacheck}_{h}$} && \multicolumn{2}{c}{$\phihat_h$}  \\
\cmidrule(lr){4-5} \cmidrule(lr){7-8} \cmidrule(lr){10-11} \cmidrule(lr){13-14} \cmidrule(lr){16-17} 
$k$
& $h$     &&  $e_\sigma$   &e.o.c.&&  $e_\velocity$   &e.o.c.&&  $e_{\phi}$   & e.o.c. && $e_{\sigmacheck}$   & e.o.c. && $e_{\phihat}$ & e.o.c. \\
\hline
\multicolumn{17}{c}{\normalsize{Triangles\phantom{$R^\big|q_|$}}}
 \\
\midrule
\multirow{6}{*}{0}
&6.25e-02 && 4.44e-01 & - && 2.36e-01 & - && 4.66e-01 & - && 4.66e+00 & - && 4.96e+00 & - \\
&3.12e-02 && 2.98e-01 & 0.58 && 1.30e-01 & 0.86 && 3.23e-01 & 0.53 && 4.40e+00 & 0.08 && 4.82e+00 & 0.04 \\
&1.56e-02 && 1.86e-01 & 0.68 && 6.99e-02 & 0.89 && 2.03e-01 & 0.67 && 3.89e+00 & 0.18 && 4.26e+00 & 0.18 \\
&7.81e-03 && 1.10e-01 & 0.77 && 3.68e-02 & 0.93 && 1.18e-01 & 0.78 && 3.23e+00 & 0.27 && 3.51e+00 & 0.28 \\
&3.91e-03 && 6.16e-02 & 0.83 && 1.90e-02 & 0.95 && 6.59e-02 & 0.84 && 2.57e+00 & 0.33 && 2.76e+00 & 0.34 \\
&1.95e-03 && 3.37e-02 & 0.87 && 9.66e-03 & 0.97 && 3.58e-02 & 0.88 && 1.99e+00 & 0.37 && 2.12e+00 & 0.38  \\\hline
\multirow{6}{*}{1} 
&5.00e-01 && 6.00e-01 & -    && 3.03e-01 & -    && 3.59e-01 & -    && 2.96e+00 & -    && 1.65e+00 & - \\
&2.50e-01 && 1.68e-01 & 1.84 && 1.09e-01 & 1.48 && 1.07e-01 & 1.74 && 1.01e+00 & 1.55 && 6.63e-01 & 1.32 \\
&1.25e-01 && 3.66e-02 & 2.20 && 3.27e-02 & 1.73 && 2.08e-02 & 2.37 && 3.06e-01 & 1.73 && 1.76e-01 & 1.91 \\
&6.25e-02 && 8.22e-03 & 2.15 && 8.80e-03 & 1.89 && 3.94e-03 & 2.40 && 9.81e-02 & 1.64 && 4.63e-02 & 1.93 \\
&3.12e-02 && 1.96e-03 & 2.07 && 2.27e-03 & 1.96 && 8.25e-04 & 2.26 && 3.31e-02 & 1.57 && 1.33e-02 & 1.80\\
&1.56e-02 && 4.81e-04 & 2.03 && 5.75e-04 & 1.98 && 1.90e-04 & 2.12 && 1.15e-02 & 1.53 && 4.17e-03 & 1.67 \\ \hline
\multirow{6}{*}{2} 
&5.00e-01 && 1.43e-01 & -    && 6.58e-02 & -    && 1.18e-01 & -    && 6.59e-01 & -    && 6.51e-01 & - \\
&2.50e-01 && 1.66e-02 & 3.11 && 1.19e-02 & 2.46 && 1.06e-02 & 3.49 && 9.73e-02 & 2.76 && 8.38e-02 & 2.96 \\
&1.25e-01 && 1.92e-03 & 3.11 && 1.55e-03 & 2.94 && 1.02e-03 & 3.38 && 1.38e-02 & 2.81 && 1.18e-02 & 2.83 \\
&6.25e-02 && 2.31e-04 & 3.06 && 1.95e-04 & 2.99 && 1.06e-04 & 3.26 && 2.18e-03 & 2.67 && 1.80e-03 & 2.71 \\
&3.12e-02 && 2.83e-05 & 3.03 && 2.45e-05 & 2.99 && 1.16e-05 & 3.19 && 3.60e-04 & 2.60 && 2.89e-04 & 2.64 \\
&1.56e-02 && 3.49e-06 & 3.02 && 3.07e-06 & 3.00 && 1.33e-06 & 3.13 && 6.11e-05 & 2.56 && 4.82e-05 & 2.58 \\ \hline
\multirow{6}{*}{3}  
&5.00e-01 && 2.58e-02 & -    && 1.42e-02 & -    && 1.30e-02 & -    && 1.35e-01 & -    && 8.24e-02 & - \\
&2.50e-01 && 1.72e-03 & 3.91 && 1.07e-03 & 3.73 && 1.08e-03 & 3.59 && 1.12e-02 & 3.59 && 9.67e-03 & 3.09 \\
&1.25e-01 && 1.01e-04 & 4.09 && 7.81e-05 & 3.77 && 5.95e-05 & 4.18 && 9.36e-04 & 3.58 && 7.58e-04 & 3.67 \\
&6.25e-02 && 5.76e-06 & 4.14 && 5.26e-06 & 3.89 && 2.85e-06 & 4.38 && 7.60e-05 & 3.62 && 5.12e-05 & 3.89 \\
&3.12e-02 && 3.36e-07 & 4.10 && 3.40e-07 & 3.95 && 1.37e-07 & 4.38 && 6.29e-06 & 3.59 && 3.43e-06 & 3.90 \\
&1.56e-02 && 2.02e-08 & 4.06 && 2.16e-08 & 3.98 && 6.79e-09 & 4.33 && 5.34e-07 & 3.56 && 2.37e-07 & 3.86 \\
\hline
\multicolumn{17}{c}{\normalsize{Squares\phantom{$R^\big|q_|$}}}
 \\
\midrule
\multirow{6}{*}{0}
&6.25e-02 && 4.86e-01 & -    && 3.15e-01 & -    && 4.86e-01 & -    && 3.85e+00 & - &  &3.85e+00 & - \\
&3.12e-02 && 3.58e-01 & 0.44 && 1.73e-01 & 0.87 && 3.58e-01 & 0.44 && 4.03e+00 & -0.07&& 4.03e+00 & -0.07 \\
&1.56e-02 && 2.39e-01 & 0.58 && 9.34e-02 & 0.89 && 2.39e-01 & 0.58 && 3.81e+00 & 0.08 && 3.81e+00 & 0.08 \\
&7.81e-03 && 1.46e-01 & 0.71 && 4.95e-02 & 0.92 && 1.46e-01 & 0.71 && 3.29e+00 & 0.21 && 3.29e+00 & 0.21 \\
&3.91e-03 && 8.39e-02 & 0.80  && 2.57e-02 & 0.94 && 8.39e-02 & 0.80  && 2.67e+00 & 0.30  && 2.67e+00 & 0.30 \\
&1.95e-03 && 4.64e-02 & 0.85 && 1.32e-02 & 0.97 && 4.64e-02 & 0.85 && 2.09e+00 & 0.35 && 2.09e+00 & 0.35 \\ \hline
\multirow{6}{*}{1} 
&5.00e-01 && 6.97e-01 & -    && 4.59e-01 & -    && 6.97e-01 & -    && 2.39e+00 & -    && 2.39e+00 & - \\
&2.50e-01 && 2.65e-01 & 1.39 && 1.19e-01 & 1.95 && 2.65e-01 & 1.39 && 1.15e+00 & 1.05 && 1.15e+00 & 1.05 \\
&1.25e-01 && 6.28e-02 & 2.08 && 3.06e-02 & 1.96 && 6.28e-02 & 2.08 && 3.43e-01 & 1.75 && 3.43e-01 & 1.75 \\
&6.25e-02 && 1.21e-02 & 2.37 && 8.06e-03 & 1.92 && 1.21e-02 & 2.37 && 8.25e-02 & 2.06 && 8.25e-02 & 2.06 \\
&3.12e-02 && 2.48e-03 & 2.29 && 2.09e-03 & 1.95 && 2.48e-03 & 2.29 && 2.02e-02 & 2.03 && 2.02e-02 & 2.03 \\
&1.56e-02 && 5.64e-04 & 2.14 && 5.31e-04 & 1.97 && 5.64e-04 & 2.14 && 5.54e-03 & 1.87 && 5.54e-03 & 1.87 \\           \hline
\multirow{6}{*}{2} 
&5.00e-01 && 1.48e-01 & -    && 1.17e-01 & -    && 1.48e-01 & -    && 5.43e-01 & -    && 5.43e-01 & - \\
&2.50e-01 && 1.76e-02 & 3.07 && 1.85e-02 & 2.66 && 1.76e-02 & 3.07 && 8.18e-02 & 2.73 && 8.18e-02 & 2.73 \\
&1.25e-01 && 2.24e-03 & 2.98 && 2.25e-03 & 3.04 && 2.24e-03 & 2.98 && 1.46e-02 & 2.48 && 1.46e-02 & 2.48 \\
&6.25e-02 && 2.77e-04 & 3.01 && 2.79e-04 & 3.01 && 2.77e-04 & 3.01 && 2.51e-03 & 2.54 && 2.51e-03 & 2.54 \\
&3.12e-02 && 3.43e-05 & 3.02 && 3.48e-05 & 3.00 && 3.43e-05 & 3.02  && 4.33e-04 & 2.54 && 4.33e-04 & 2.54 \\
&1.56e-02 && 4.26e-06 & 3.01 && 4.35e-06 & 3.00 && 4.26e-06 & 3.01  && 7.52e-05 & 2.52 && 7.52e-05 & 2.52 \\    \hline
\multirow{6}{*}{3}  
&5.00e-01 && 2.62e-02 & -    && 2.42e-02 & -    && 2.62e-02 & -    && 1.13e-01 & -    && 1.13e-01 & - \\
&2.50e-01 && 2.13e-03 & 3.62 && 1.82e-03 & 3.73 && 2.13e-03 & 3.62 && 1.28e-02 & 3.14 && 1.28e-02 & 3.14 \\
&1.25e-01 && 1.30e-04 & 4.04 && 1.46e-04 & 3.64 && 1.30e-04 & 4.04 && 1.05e-03 & 3.61 && 1.05e-03 & 3.61 \\
&6.25e-02 && 7.61e-06 & 4.09 && 1.10e-05 & 3.73 && 7.61e-06 & 4.09 && 7.53e-05 & 3.80  && 7.53e-05 & 3.80 \\
&3.12e-02 && 4.54e-07 & 4.07 && 7.49e-07 & 3.87 && 4.54e-07 & 4.07 && 5.36e-06 & 3.81 && 5.36e-06 & 3.81 \\
&1.56e-02 && 2.76e-08 & 4.04 && 4.87e-08 & 3.94 && 2.76e-08 & 4.04 && 4.00e-07 & 3.75 && 4.00e-07 & 3.75 \\
\bottomrule
\end{tabular}
\caption{Experiment \ref{sec: Experiment3}. Dirichlet boundary conditions.}
\label{tab:HistoryConvergenceExperiment3}
\end{table}
\section{Concluding remarks and future work} \label{section:conclussions_and_ongoing_work}
Compared with mixed finite element methods \cite{23_Arnold,18_HWang}, the proposed HDG methods improve the order convergence of the rotational and divergence variables from $k-1/2$ to $k+1/2$. in the case of Dirichlet boundary conditions. Numerical results for electric and magnetic boundary conditions were consistent with those in \cite{23_Arnold,18_HWang}. Note that, in contrast to the FEEC-based hybrid methods of \cite{14_JGuzman}, the present HDG formulation avoids compatibility constraints associated with discrete De Rham complexes while at least the same
accuracy. Moreover, unlike the nonconforming hybrid methods of \cite{6_Barker,7_Barker}, the HDG directly controls both the rotational and divergence variables and yielded a symmetric positive-definite system.

We are interested in improving the theoretical convergence rates, motivated by the numerical results presented in this work and by extending ideas from \cite{8_Cockburn,37_Shukai,38_Oikawa}. Furthermore, it is important to note that although we primarily report results using standard diagonal triangular meshes, similar optimal behavior was observed for other types of quasi-uniform grids, such as Chevron, Union Jack, and Peterson-type meshes. This objective is particularly relevant due to the connections, arising in the design of suitable projections for the error analysis, between our HDG framework and methodologies developed for Maxwell equations \cite{40_MaxwellHDG,9_MSolano,41_Cockburn,37_Shukai}, the Stokes system \cite{36_Cui,34_Cockburn,3_BCockburn,28_Cockburn}, eddy current problems \cite{31_Osorio}, and shallow water equations \cite{MR3576573}. We are also interested in examining the role of the $M$-decomposition for the vector Laplacian, which can facilitate the construction of approximation spaces that yield optimal convergence in more general HDG formulations (see, e.g., \cite{1P_Monk,21_Cockburn,22_Cockburn,20_Cockburn,2_BCockburn}). Finally, we want to extend the method to three dimensions.

\appendix
\section{Characterization of the hybridizations}
\label{sec:Appendix}
In this section, we present the characterizations of the Hybridizations of Type I and Type II. The proofs are similar to that given in Section \ref{section:Proofs_Hybridization} and are therefore omitted.
\subsection*{Hybridization of Type I}

\begin{theoremn} \label{Teo:CharacterizationHDGsolutionrestI} 
 The solution of the Hybridization of Type I method verifies 
\begin{alignat*}{4}
(\sigma_h,\phi_h,\velocity_h)=(\Sigma,\Phi,\vect{U})=(\Sigmasecond,\Phisecond,\Usecond)+(\Sigmasecondf,\Phisecondf,\Usecondf),    
\end{alignat*}
where, on the element $K\in \Th,$ for any $\etab^t\in \Lb^2(\partial K),$ $\lambda\in L^2(\partial K)$, the function $(\Sigmasecondeta,\Phisecondeta,\Usecondeta)$ in $\PK\times \PK \times \PKb$ is the solution of the following local problem
\begin{subequations} 
\begin{alignat*}{4} 
\Kprod{\Sigmasecondeta}{\chi}-\Kprod{\Usecondeta}{\curl \, \chi}+ \pKprod{ \etab^t \cdot \normal^\perp}{\chi
} & =0,   \quad \forall \chi \in \PK,\\  
\Kprod{\Phisecondeta}{\varphi}-\Kprod{\Usecondeta}{\grad\varphi}+\pKprod{\widehat{\Ub}_{\etab,\lambda}\cdot \normal}{\varphi
} & =0, \quad \forall \varphi \in \PK ,& 
\\ 
\Kprod{\Sigmasecondeta}{\rot \, \zb}+\pKprod{\widecheck{\Sigma}_{\etab^t} }{\zb\cdot \normal^\perp
}   - 
\Kprod{\Phisecondeta}{\divergence \zb}
+\pKprod{\lambda}{\zb \cdot \normal}
& = 0, \quad \forall \zb \in \PKb,
\end{alignat*}
where 
\begin{alignat*}{4}
\widecheck{\Sigma}_{\etab^t}&=\Sigmasecondeta+\alpha^{-1} (\Usecondeta \cdot \normal^\perp -\etab^t \cdot \normal^\perp), & \text{ on }\partial K,\\
\widehat{\Ub}_{\etab,\lambda}\cdot \normal&=\Usecondeta \cdot \normal+\tau (\Phisecondeta-\lambda), & \text{ on }\partial K.
\end{alignat*}
Besides, for any $\fb\in \Lb^2(K)$, then $(\Sigmasecondf,\Phisecondf,\Usecondf)$ in $\PK\times \PK \times \PKb$ is the solution of  
\begin{alignat*}{4} 
\Kprod{\Sigmasecondf}{\chi}-\Kprod{\Usecondf}{\curl \, \chi} & =0, & \forall \chi \in \PK,\\  
\Kprod{\Phisecondf}{\varphi}-\Kprod{\Usecondf}{\grad\varphi}+\pKprod{\widehat{\Ub}_{\fb}\cdot \normal}{\varphi
} & =0, & \forall \varphi \in \PK ,& 
\\ 
\Kprod{\Sigmasecondf}{\rot \, \zb}+\pKprod{\widecheck{\Sigma}_{\fb} }{\zb\cdot \normal^\perp
}   - 
\Kprod{\Phisecondf}{\divergence \zb}
& = \Kprod{\fb}{ \zb}, & \quad  \forall \zb \in \PKb,
\end{alignat*}
where
$ 
\widecheck{\Sigma}_{\fb}=\Sigmasecondf+\alpha^{-1} \Usecondf \cdot \normal^\perp,  \text{ on }\partial K$,
and 
$\widehat{\Ub}_{\fb}\cdot \normal=\Usecondf \cdot \normal+\tau \Phisecondf, \text{ on }\partial K.$
The pair $(\widecheck{\velocity}_h,\widehat{\phi}_h)\in \Nhperp \times \Mh$ is the solution of
\begin{alignat*}{4} 
a_h((\widecheck{\velocity}_h,\widehat{\phi}_h),(\mub^t,\mu))=l_h(\mub^t,\mu),
\end{alignat*}
with 
\begin{alignat*}{4} \nonumber   
a_h((\etab^t,\lambda),(\mub^t,\mu))&= \Thprod{\Sigmasecondeta}{\Sigmasecondmu}+\pThprod{\alpha^{-1}(\Usecondeta \cdot \normal^\perp - \etab^t \cdot \normal^\perp)}{(\Usecondmu \cdot \normal^\perp- \mub \cdot \normal ^\perp)} \\
&\quad + \Thprod{\Phisecondeta}{\Phisecondmu}  +\pThprod{\tau (\Phisecondeta - \lambda)}{(\Phisecondmu -\mu)}, \\
l_h(\mub^t,\mu)&= \Thprod{\fb}{\Usecondmu},  
\end{alignat*}
and $(\widecheck{\velocity}_h,\widehat{\phi}_h)$ minimizes the functional of energy
$\mathcal{J}_h(\mub^t,\mu)=\frac{1}{2}a_h((\mub^t,\mu),(\mub^t,\mu))-l_h(\mub^t,\mu)$
over the functions $\mub^t \in \Nhperp $ and $\mu \in \Mh $.
\end{subequations}
\end{theoremn}

\subsection*{Hybridization of Type II}

\begin{theoremn} \label{Teo:CharacterizationHDGsolutionrestII} 
The solution of Hybridization of Type II satisfies
\begin{alignat*}{4}
(\sigma_h,\phi_h,\velocity_h)=(\Sigma,\Phi,\vect{U})=(\Sigmathird,\Phithird,\Uthird)+(\Sigmathirdf,\Phithirdf,\Uthirdf),    
\end{alignat*}
where, on the element $K\in \Th,$ for any $\varsigma\in L^2(\partial K)$ and $\etab^n\in \Lb^2(\partial K),$ the function $(\Sigmathirdeta,\Phithirdeta,\Uthirdeta)$ in $ \PK\times \PK \times \PKb$ is the solution of the following local problem
\begin{subequations} 
\begin{alignat*}{4} 
\Kprod{\Sigmathirdeta}{\chi}-\Kprod{\Uthirdeta}{\curl \, \chi}+ \pKprod{ \widecheck{\Ub}_{\varsigma,\etab^n} \cdot \normal^\perp}{\chi
} & =0,   \quad \forall \chi \in \PK,\\  
\Kprod{\Phithirdeta}{\varphi}-\Kprod{\Uthirdeta}{\grad\varphi}+\pKprod{\etab^n\cdot \normal}{\varphi
} & =0, \quad \forall \varphi \in \PK ,& 
\\ 
\Kprod{\Sigmathirdeta}{\rot \, \zb}+\pKprod{\varsigma }{\zb\cdot \normal^\perp
}   - 
\Kprod{\Phithirdeta}{\divergence \zb}
+\pKprod{\widehat{\Phi}_{\etab^n}}{\zb \cdot \normal}
& = 0, \quad \forall \zb \in \PKb,
\end{alignat*}
where 
\begin{alignat*}{4}  
\widecheck{\Ub}_{\varsigma,\etab^n} \cdot \normal^\perp&=\Uthirdeta \cdot \normal^\perp+\alpha (\Sigmathirdeta-\varsigma), & \text{ on }\partial K,\\ 
\widehat{\Phi}_{\etab^n}&=\Phithirdeta+\tau^{-1} (\Uthirdeta \cdot \normal -\etab^n \cdot \normal), & \text{ on }\partial K.
\end{alignat*}
Furthermore, we have the following,  for any $\fb\in \Lb^2(K),$ then $(\Sigmathirdf,\Phithirdf,\Uthirdf)$ in $\PK\times \PK \times \PKb$ is the solution of  
\begin{alignat*}{4}
\Kprod{\Sigmathirdf}{\chi}-\Kprod{\Uthirdf}{\curl \, \chi} + \pKprod{\widecheck{\Ub}_\fb \cdot \normal^\perp}{\chi} & =0, & \forall \chi \in \PK,\\  
\Kprod{\Phithirdf}{\varphi}-\Kprod{\Uthirdf}{\grad\varphi} & =0, & \forall \varphi \in \PK ,& 
\\ 
\Kprod{\Sigmathirdf}{\rot \, \zb} - 
\Kprod{\Phithirdf}{\divergence \zb}
+\pKprod{\widehat{\Phi}_\fb}{\zb \cdot \normal} & = \Kprod{\fb}{ \zb}, & \quad \forall \zb \in \PKb,
\end{alignat*}
where
$
\widecheck{\Ub}_{\fb}\cdot \normal^\perp=\Uthirdf \cdot \normal^\perp+\alpha \Sigmathirdf,  \text{ on }\partial K,$ and $
\widehat{\Phi}_{\fb}=\Phithirdf+\tau^{-1} \Uthirdf \cdot \normal,  \text{ on }\partial K.
$
We have that $(\widecheck{\sigma}_h,\widehat{\velocity}_h)\in \Mh   \times \Nh$ is the solution of
\begin{alignat*}{4}  
a_h((\widecheck{\sigma}_h,\widehat{\velocity}_h),(\mu,\mub^n))=l_h(\mu,\mub^n),
\end{alignat*}
with 
\begin{alignat*}{4}    
a_h((\varsigma,\etab^n),(\mu,\mub^n))&= \Thprod{\Sigmathirdeta}{\Sigmathirdmu}
+\pThprod{\alpha(\Sigmathirdeta - \varsigma)}{(\Sigmathirdmu - \mu)} \\
&\quad +\Thprod{\Phithirdeta}{\Phithirdmu}+\pThprod{\tau^{-1} (\Uthirdeta \cdot \normal  - \etab^n\cdot \normal)}{(\Uthirdmu \cdot \normal  - \mub^n\cdot \normal)}, \\
l_h(\mu,\mub^n)&= \Thprod{\fb}{\Uthirdmu}, 
\end{alignat*}
and $(\widecheck{\sigma}_h,\widehat{\velocity}_h)$ minimizes the functional of energy
$\mathcal{J}_h(\mu,\mub^n)=\frac{1}{2}a_h((\mu,\mub^n),(\mu,\mub^n))-l_h(\mu,\mub^n)$ 
over the functions $\mu \in \Mh $ and $\mub^n \in \Nh $.
\end{subequations}
\end{theoremn} 
\bibliography{bibliography}{}
\bibliographystyle{siam}

\end{document}

%% file: packages.tex
\usepackage[foot]{amsaddr}
\usepackage{amssymb, amsbsy, exscale, amsmath, amsfonts, amscd}
\usepackage{geometry}   
\usepackage{tabularx}
\newcolumntype{C}{>{\centering\arraybackslash}X}
\usepackage{graphicx}
\usepackage{booktabs}
\usepackage{epstopdf}
\usepackage{bm}
\usepackage{threeparttablex}
\usepackage{multicol}
\usepackage{multirow}
\usepackage{subfigure}
\usepackage{tikz}
\usepackage{enumitem}
\usepackage{array}
\let\capstd\cap   %
\usepackage{mathabx}
\let\cap\capstd 
\usepackage{floatrow}            
\usepackage{mymacros} 
\usepackage{pdfsync}
\usepackage{hyperref}
\usepackage{algpseudocode} 
\usepackage[ruled,vlined]{algorithm2e}
\usepackage{wrapfig}
\usepackage{flushend}
\usepackage{xcolor}
\usepackage{pifont}
\usepackage{todonotes}
\usepackage{pgfplots}
\usepgfplotslibrary{groupplots}
\usepgfplotslibrary{colorbrewer}
\usepackage{nameref}
\pgfplotsset{compat=newest}
\DeclareCaptionLabelFormat{AppendixTables}{Table A.#2}
\DeclareCaptionLabelFormat{AppendixTables2}{Table B.#2}
\newsavebox{\measurebox}

\usepackage{lipsum}
\usepackage[skip=0.33\baselineskip]{caption}
\usepackage{adjustbox}
\usepackage{mwe}
\usepackage[most]{tcolorbox}
\usepackage{float}

\usepackage{booktabs}
\usepackage{tabularx}
\usepackage{array}

\newcolumntype{Y}{>{\raggedright\arraybackslash}X}